\documentclass[10pt]{article}
\usepackage{color}
\usepackage{amssymb}
\usepackage{amsthm,array,amssymb,amscd,amsfonts,latexsym, url}
\usepackage{amsmath}
\usepackage[all]{xy}
\newtheorem{theo}{Theorem}[section]
\newtheorem{prop}[theo]{Proposition}
\newtheorem{claim}[theo]{Claim}
\newtheorem{lemm}[theo]{Lemma}
\newtheorem{coro}[theo]{Corollary}
\newtheorem{rema}[theo]{Remark}

\newtheorem{conj}[theo]{Conjecture}

\title{On the  Chow ring of very general  abelian varieties and a question of Pirola}
\author{Claire Voisin\footnote{The author is supported by the ERC Synergy Grant HyperK (Grant agreement No. 854361).}
}
\date{}

\newfont{\gothic}{eufb10}
\begin{document}

\maketitle
\begin{abstract}  We prove that for a very  general abelian variety of dimension $\geq 4$, a divisor $D\in {\rm CH}^1(A)$ that satisfies $D^2=0$ in ${\rm CH}^2(A)$ is of torsion. The same result is also established for a very general Jacobian in genus $4$. We use then  the second statement in order to prove a conjecture of Pirola, which  states that any rational section of the Kummer fibration $K=J/\pm {\rm Id}\rightarrow \mathcal{M}_4$, where $J\rightarrow \mathcal{M}_4 $ is the Jacobian fibration, must be a multiple of the Griffiths-Pirola section given by the difference of the two trigonal divisors.
 \end{abstract}
\section{Introduction}
Let $A$ be an abelian variety over $\mathbb{C}$. The study of the Chow ring ${\rm CH}^*(A) $ with $\mathbb{Q}$-coefficients has been initiated by Bloch \cite{blochabelian} and  Beauville \cite{beauville} who discovered beautiful structures like  the Beauville decomposition \cite{beauville}, \cite{beauville1} and  the Fourier-Mukai transform \cite{mukai} induced by the abelian group structure of $A$ on one hand,    and the isogeny between $A$ and its dual on the other hand. In \cite{voisinchowab}, we started the study of more  basic questions about  the ring structure of ${\rm CH}^*(A)$. (The paper \cite{voisinchowab} focused on the Pontryagin powers of $0$-cycles, but via the Fourier-Mukai transform, this can be reformulated using only the ordinary ring structure of ${\rm CH}^*(A)$ (see Section \ref{secnewconspont}.) The basic question is then the following: given $r$,  what can be said about the set of  divisors $D\in {\rm CH}^1(A)_{\rm hom}$ such that  $D^r=0$  in ${\rm CH}^r(A)$? If $r\geq2$, this set is a countable union of closed algebraic subsets of ${\rm Pic}^0(A)$ which contains the torsion points of ${\rm Pic}^0(A)$. In {\it loc. cit.}, it was proved that for a very general abelian variety of dimension $\geq 3$, the set of divisors $D\in {\rm Pic}^0(A)$ such that $D^2=0$ in ${\rm CH}^2(A)$
has dimension $0$.
In the papers \cite{colombopirola}, the locus of abelian varieties for which this statement fails is studied in a very precise way : Away  from decomposable abelian varieties, this locus has dimension $\leq 2g-1$ (so in particular, one recovers the fact that this locus is not the whole of $\mathcal{A}_g$ when $g\geq 3$).

Note that one can ask similar questions for symmetric products  of curves, or for  their Jacobians. Let $C$ be a smooth projective  curve and let  $j: C\hookrightarrow  J(C)$ be an embedding. The pull-back map
$$j^*: {\rm Pic}^0(J(C))\rightarrow  {\rm Pic}^0(C)$$
is an isomorphism, which does not depend on the choice of $j$. Note that we also have the isomorphism
${\rm Pic}^0(C)\cong {\rm Pic}^0(C^{(2)})$, and we can consider, for any $$D\in {\rm Pic}^0(C)\cong {\rm Pic}^0(C^{(2)})\cong {\rm Pic}^0(J(C))$$ the codimension $2$ cycles $D^2\in {\rm CH}^2(J(C))$ and   $D^2\in {\rm CH}^2(C^{(2)})$. It turns out that the vanishing of these two squares are equivalent:
\begin{lemm} \label{lenouveau2804} One has $D^2=0$ in  ${\rm CH}^2(J(C))$ if and only if  $D^2=0$ in ${\rm CH}^2(C^{(2)})$.
\end{lemm}
\begin{proof} Clearly the first vanishing implies the second one by restriction from $J(C)$ to $C^{(2)}$. In the other direction, assume that $D^2$ vanishes in ${\rm CH}_0(C^{(2)})$. Then  by Bloch's formula \cite[Lemma 1.1)(ii)]{blochabelian}, one deduces that $s^{*2}=0$ in ${\rm CH}_0(J(C))$ where $s:=D\cdot C\in {\rm CH}_0(J(C))$ and $*$ is the Pontryagin product. Denoting by $F$  the Fourier-Mukai transform of $J(C)$ (see \cite{beauville}) it follows that
\begin{eqnarray}\label{eqintro} F(D\cdot C)\cdot  F(D\cdot C)=0\,\,{\rm in}\,\,{\rm CH}^*(J(C))\end{eqnarray}
since $F$ exchanges the Pontryagin product and the intersection product. Equation (\ref{eqintro}) then implies $D^2=0$ in ${\rm CH}^2(J(C))$ because $F(D\cdot C)=D+ \Gamma$ where $\Gamma$ is a sum of cycles of codimension $\geq2$.
\end{proof}

Our first main result in the present paper is the following
\begin{theo}\label{theoabelian} (i) Let $A $ be a very general abelian variety of dimension at least $4$. Then, if $D\in {\rm CH}^1(A)_{\rm hom}={\rm Pic}^0(A)$, and $D^2=0$ in ${\rm CH}^2(A)$, $D$ is of torsion.

(ii)  Let $C$ be a very general curve of genus $ 4$. Then if $D\in {\rm Pic}^0(C)={\rm Pic}^0(C^{(2)})$ and $D^2=0$ in ${\rm CH}^2(C^{(2)})$ (or equivalently by Lemma \ref{lenouveau2804}, $D^2=0$ in ${\rm CH}^2(J(C))$), $D$ is of torsion.
\end{theo}
  Thanks to the Fourier-Mukai transform acting as above on Chow groups, one can also rephrase Theorem \ref{theoabelian}  using the Pontryagin product of $0$-cycles on $A$. This will be done in  Section \ref{secnewconspont}, where we also give, following   \cite{voisinchowab}, a consequence on orbits for rational equivalence of effective zero-cycles on $A$.
\begin{rema}\label{remaexampleE} {\rm The very general assumption is necessary in Theorem \ref{theoabelian}. The simplest examples of abelian varieties not satisfying the conclusion of Theorem \ref{theoabelian}(i) are those admitting a nonconstant morphism $\phi:A\rightarrow E$, where $E$ is an elliptic curve.  Indeed, $\phi^*:{\rm Pic}^0(E)\rightarrow {\rm Pic}^0(A)$ has a finite kernel and for any $D\in {\rm Pic}^0(E)$, one has $(\phi^*D)^2=0$ in ${\rm CH}^2(A)$. The points parameterizing such abelian varieties $A$ are analytically  dense in the moduli space.}
\end{rema}
\begin{rema}{\rm It is likely that Theorem \ref{theoabelian}(ii) also holds for any genus $g\geq4$. The extra work needed to make the same proof work in higher genus is to prove  Proposition \ref{propcarrenul}(ii) in higher genus. However, the conclusion  of Theorem \ref{theoabelian}(ii) is wrong  for $g=2$  and more generally for Jacobians of hyperelliptic curves (see Remark \ref{remag2du910}).}
\end{rema}

\begin{rema}{\rm It is likely that the method used to prove Theorem \ref{theoabelian} also allows to study higher powers. One may even conjecture that if $g\geq 2k$, and $A$ is a very general abelian variety of dimension $g$,  the only divisors $D\in {\rm CH}^1(A)$ such that $D^k=0$ in $ {\rm CH}^k(A)$ are the torsion divisors.}
\end{rema}
Our motivation to prove Theorem \ref{theoabelian} was   an intriguing and seemingly unrelated  question of Pirola related to the Franchetta conjecture for generic curves. The Franchetta conjecture (now a theorem, see \cite{mestrano}) says that if $C$ is the generic curve of genus $g$, then ${\rm Pic}(C)$ is the cyclic group generated by $K_C$. A weaker version says that, over the moduli space $\mathcal{M}_g$ of curves of genus $g$, the only rational section of the jacobian fibration $J_g\rightarrow \mathcal{M}_g$ is the zero-section. We  study here  the rational sections of the associated Kummer fibration $K_g:=J_g/\pm{\rm Id}\rightarrow \mathcal{M}_g$.
\begin{rema}\label{remakollar} {\rm Rational sections of $K_g\rightarrow \mathcal{M}_g$ are the same as bisections of $J_g\rightarrow \mathcal{M}_g$. Indeed, if we have such a bisection $\sigma: M'_g\subset J_g$, where $M'_g$ is $2$-to-$1$ over $\mathcal{M}_g$, then,  by the weak version of the  Franchetta conjecture stated above, for a general point  $b\in \mathcal{M}_g$, its inverse image $\sigma_1(b),\,\sigma_2(b)\in  M'_{g,b}\subset  J(C_b)$ must satisfy $\sigma_1(b)+\sigma_2(b)=0$ in $J(C_b)$. Hence $\sigma$ induces a rational section of $K_g\rightarrow  \mathcal{M}_g$. The other direction is clear.}
\end{rema}
If one considers  a general curve of genus $4$ over $\mathbb{C}$, it is naturally embedded in a smooth $2$-dimensional quadric in its canonical embedding, hence it has two  $g_3^1$'s, say  $D_1,\,D_2$, coming from the restrictions of the two rulings of the quadric to $C$.
The difference $D_1-D_2\in {\rm Pic}^0(C)$ is a divisor of degree $0$ on $C$, nontrivial in general, which is canonically defined  up to sign, so for the generic curve of genus $4$, which is defined over $\mathbb{C}(\mathcal{M}_4)$, it provides an element of ${\rm Pic}^0(C_L)$ defined over a degree $2$-field extension $L$ of the function field $\mathbb{C}(\mathcal{M}_4)$. Equivalently, it provides a relative divisor on the universal curve $\widetilde{\mathcal{C}}\rightarrow \widetilde{\mathcal{M}}_4$ for some  degree $2$ cover $\widetilde{\mathcal{M}}_4\rightarrow \mathcal{M}_4$.

We will call the corresponding  section $\gamma_P$ of ${\rm Pic}^0(\widetilde{\mathcal{C}}/\widetilde{\mathcal{M}}_4)$ the Griffiths-Pirola normal function. It has been studied indeed in \cite{griffiths}  by Griffiths, who explained how to recover the general curve $C$ of genus $4$ from the infinitesimal invariant of this normal function at $[C]\in  \mathcal{M}_4$ (see also \cite{verni} and Section \ref{secreview}). We can also see $\gamma_P$ as a rational section of the Kummer fibration ${K}={J}/\pm {\rm Id}\rightarrow \mathcal{M}_4$.
 The following  result answers a question of Pirola.
\begin{theo}\label{theoquestionpirola} Let $\mathcal{C}\rightarrow  \mathcal{M}_4$ be the universal curve of genus $4$, and ${J}\rightarrow \mathcal{M}_4 $ the associated Jacobian fibration. Then all rational sections of the fibration ${K}={J}/\pm {\rm Id}\rightarrow \mathcal{M}_4$ are obtained  by taking  multiples of the  Griffiths-Pirola normal function $\gamma_P$. Equivalently by Remark \ref{remakollar}, any bisection of the Jacobian fibration $J_4\rightarrow \mathcal{M}_4$ is an integral multiple of the Griffiths-Pirola section.
\end{theo}
Note that a  result related to Theorem \ref{theoquestionpirola} had been  established in \cite{hangxue}, where it is proved that, on the given double cover $\widetilde{\mathcal{M}}_4\rightarrow \mathcal{M}_4$, the only sections of the pulled-back Jacobian fibration $\widetilde{J}\rightarrow \widetilde{\mathcal{M}}_4$ are the multiples of  ${\gamma}_P$.

Theorem \ref{theoquestionpirola} suggests the following conjecture  (proposed by Pirola and Farkas).
\begin{conj} For $g\geq 5$, the only rational section of the Kummer fibration  ${K}={J}/\pm {\rm Id}\rightarrow \mathcal{M}_g$ is the zero section.
\end{conj}
\subsection{Strategy of the proofs}
 The main difficulty in Pirola's question, which makes it rather different from Franchetta's conjecture,  is the fact that the rational  sections $\gamma:\mathcal{M}_4\dashrightarrow K$ that we want to describe correspond to  sections of the Jacobian fibration  over a double cover of $\mathcal{M}_4$ that is not specified. The link between Theorem \ref{theoabelian}(ii) and  Theorem \ref{theoquestionpirola} is the following observation : Given a section $\gamma$ of the Kummer fibration, as the divisor $D_\gamma(t)\in {\rm Pic}^0(C_t)={\rm Pic}^0(C_t^{(2)})$ is defined up to sign, the codimension $2$ cycle $D_\gamma(t)^2\in {\rm CH}^2(C_t^{(2)})$ is  defined  over the function field of $\mathcal{M}_4$, that is, comes from a codimension $2$ cycle (with $\mathbb{Q}$-coefficients) on $C_\eta^{(2)}$, where $\eta$ is the generic point of $\mathcal{M}_4$.  One easily deduces from this that it is a multiple of $D_{\gamma_P}(t)^2$ (which is  proportional to the Green-Griffiths $0$-cycle of \cite{greengriffiths}, also studied in \cite{yinq} and   \cite{verni}). In other words, we get a relation for any $t\in \mathcal{M}_4$
\begin{eqnarray}\label{eqnewdu910} D_\gamma^2=\lambda D_{\gamma_P}^2\,\,{\rm in}\,\,{\rm CH}^2(C_t^{(2)}).\end{eqnarray}

When $\lambda=0$, Theorem \ref{theoabelian}(ii) directly implies that the section $\gamma$ is of torsion, and one easily concludes that it actually vanishes. In general, our proof consists  in analyzing, using  the theory of infinitesimal invariants as described below and in Section \ref{secreview}, the rational equivalence relation (\ref{eqnewdu910}), and deduce that $\lambda={\lambda'}^2$ is the square of an integer  and that $D_\gamma=\lambda' D_{\gamma_P}$ in ${\rm CH}^1(C_t^{(2)})={\rm CH}^1(C_t)$.

The proof of Theorem \ref{theoabelian} uses the theory of Griffiths infinitesimal invariants $\delta\nu$ of normal functions $\nu$ \cite{griffiths}, and  its  extension to higher codimensional cycles introduced  in \cite{voisinsurfaces}, (where rational equivalence of $0$-cycles of very general surfaces in $\mathbb{P}^3$ was studied,) which  also  appears implicitly  in \cite{nori}  (see Section \ref{secreview}). Given a codimension $k$ cycle $Z\subset \mathcal{X}$, where $\mathcal{X}$ is smooth and admits a smooth projective morphism $\phi:\mathcal{X}\rightarrow B$, with $B$ smooth, the infinitesimal invariant of $Z$ at any point $b\in B$ is an element  $\delta Z(b)\in H^k(\mathcal{X}_b,\Omega_{\mathcal{X}\mid \mathcal{X}_b}^k)$, which depends only on the restriction $Z_1\subset \mathcal{X}_1:=\phi^{-1}(b_1)$ over   the first infinitesimal neighborhood of $b$ in $B$. Our starting point is the easy but important fact that, if the restriction $Z_b:=Z_{\mid \mathcal{X}_b}$ vanishes in ${\rm CH}^k (\mathcal{X}_b)$ for any $b\in B$, then the infinitesimal invariant $\delta Z(b)$ vanishes at a general point  $b\in B$ (see Proposition \ref{theomumfordvar}). Given a family of divisors
$D_b\in {\rm Pic}^0(J_{g,b})$ parameterized by a generically finite cover $\mathcal{A}'_g$ of $\mathcal{A}_g$, such that
$D_b^2=0$ in ${\rm CH}^2(J_{g,b})$, we thus conclude that
\begin{eqnarray}\label{eqnewdu1010}\delta (D^2)(b)=0
\end{eqnarray} at a general point $b\in \mathcal{A}'_g$. It is clear from the definition of the infinitesimal invariants that one has the relation
 $$\delta (D^2)(b)=(\delta D(b))^2\,\,{\rm in}\,\,H^2(J_{g,b},\Omega_{J_{g}\mid J_{g,b}}^2).$$
 The square map and more generally the quadratic map
$${\rm Sym}^2 H^1(J_{g,b},\Omega_{J_{g}\mid J_{g,b}})\rightarrow  H^2(J_{g,b},\Omega_{J_{g}\mid J_{g,b}}^2)$$
appearing in   this formula  is completely explicit but not so easy to analyze.
 In any case,  one gets from (\ref{eqnewdu1010}) the relation
\begin{eqnarray}\label{eqcarredu1010}  (\delta D(b))^2=0\,\,{\rm in}\,\,H^2(J_{g,b},\Omega_{J_{g}\mid J_{g,b}}^2).\end{eqnarray}
A first step in the analysis of (\ref{eqcarredu1010}) is presented in  Proposition \ref{propcarrenul}.

The second  ingredient in our proof is the fact (proved in \cite{voisinincras}) that the geometric infinitesimal invariant considered above coincides for codimension $1$ cycles (homologous to zero on fibers) with the Griffiths infinitesimal invariant of normal functions. The viewpoint of Griffiths on the  algebraic section of ${\rm Pic}^0(\mathcal{X}/B)$, given by the data $(D_b)_{b\in B}$ of divisors homologous to zero on $\mathcal{X}_b$, is transcendental.  Uniformization gives an isomorphism  ${\rm Pic}^0(\mathcal{X}_b)\cong H^1(\mathcal{X}_b,\mathbb{C})/(H^{1,0}(\mathcal{X}_b)\oplus H^1(\mathcal{X}_b,\mathbb{Z}))$. Griffiths  then
 considers the corresponding  section $\nu_D\in \Gamma(\mathcal{J})$,  where
 $\mathcal{J}=\mathcal{H}^1/(\mathcal{H}^{1,0}\oplus H^1_\mathbb{Z})$ is the sheaf of holomorphic sections of ${\rm Pic}^0(\mathcal{X}/B)$,   $\mathcal{H}^1$ being here  the flat vector bundle $R^1\phi_*\mathbb{C}\otimes \mathcal{O}_B$.    Griffiths' infinitesimal invariant $\delta\nu_D$ is obtained by  choosing local lifts $\tilde{\nu}$ of $\nu_D$ to a section of $\mathcal{H}^1$, and by taking the $(0,1)$-part of  $\nabla\tilde{\nu}$ in $\Gamma(\mathcal{H}^{0,1}\otimes \Omega_B)$, modulo $\overline{\nabla}(\mathcal{H}^{1,0})$. (Here, $\nabla$ is the Gauss-Manin connection of $\mathcal{H}^1$.)  From this description of
 $\delta D=\delta \nu_D$, one easily gets that, under some assumptions that are satisfied by the universal family $J_g\rightarrow \mathcal{A}_g$ (or any similar family obtained after passing to a generically finite cover $\mathcal{A}'_g$), the divisor $D_b$ is of torsion at any $b\in B$ if  and only if the infinitesimal invariant $\delta D(b)$ vanishes at the general point $b\in B$ (see Lemma \ref{lepourlatorsion}).

Coming back to the  relative divisor $(D_{b})_{b\in \mathcal{A}'_g}$ on  $J'_g\rightarrow \mathcal{A}'_g$, we cannot deduce directly from (\ref{eqcarredu1010}) that the infinitesimal invariant $\delta D(b)$ vanishes at the general point $b\in  \mathcal{A}'_g$. What we can deduce from the formal and local analysis of (\ref{eqcarredu1010}) is Lemma \ref{lepourproppourcruc} saying that
the normal function $\nu_D$ admits local  lifts
$\tilde{\nu}\in \Gamma(\mathcal{H}^1)$, which are unique up to a (flat) section of $H^1_\mathbb{Z}$, with the property that
\begin{eqnarray}\label{eqnablatildenurank1intro} \nabla{\tilde{\nu}}=\beta\otimes \alpha\,\,{\rm in}\,\,\Gamma(\mathcal{H}^1\otimes \Omega_{A'_g}),
\end{eqnarray}
that is, $\nabla{\tilde{\nu}}$ is a locally decomposable section  of $\mathcal{H}^1\otimes \Omega_{A'_g}$.
We also observe that $\nabla{\tilde{\nu}}$ is in fact a globally defined section of  $\mathcal{H}^1\otimes \Omega_{A'_g}$, which is both closed and algebraic (see Lemma \ref{lealgbetaomega}).
The end of the proof is then Proposition \ref{lecructresgenpourpropourcrcutheoab}, which tells that for $g\geq 2$, an algebraic section of  $\mathcal{H}^1\otimes \Omega_{A'_g}$, which is  closed and  locally decomposable, vanishes identically. This last  statement  is obtained as an application of the beautiful  theorem   of Ax-Shanuel type proved in \cite{nagloo}. Putting things together, we conclude that $\delta D(b)=0$ at the general point $b$ and thus that $D_b$ is torsion at any $b\in B$.

\vspace{0.5cm}

{\bf Thanks.} {\it This work was inspired by lectures of Gian Pietro Pirola at the workshop ``Topics in Hodge theory'' at the university of Ashoka. I thank the organizers of this  nice event. I also thank Greg Baldi for   the  proof of Proposition \ref{lecructresgenpourpropourcrcutheoab}, and Bruno Klingler and Gavril Farkas for their  interest and constructive  comments.}
\section{A review of infinitesimal invariants\label{secreview}}
Infinitesimal invariants of normal functions were first introduced by Griffiths (see \cite{griffiths}).
Assume that we have an abelian fibration $\pi: J\rightarrow M$. Let $D$ be a divisor on $J$, which is homologous to $0$ along the fibers of $\pi$. There is an associated
 section $\nu_D$  of  the sheaf $\mathcal{J}^{\vee} $ of holomorphic sections of the dual abelian fibration ${J}^{\vee}:={\rm Pic}^0(J/M)\rightarrow M$, such that
 $\nu_D(m):=D_{\mid J_m}$. By uniformization, this sheaf fits into an exact sequence
\begin{eqnarray}\label{eqexcatpourJ}  0\rightarrow H^{1}_\mathbb{Z}\oplus \mathcal{H}^{1,0}\rightarrow \mathcal{H}^1\rightarrow \mathcal{J}^{\vee}\rightarrow 0,
\end{eqnarray}
where $H^1_\mathbb{Z}:=R^1\pi_*\mathbb{Z},\, \mathcal{H}^1:=H^1_\mathbb{Z}\otimes\mathcal{O}_M$, and $\mathcal{H}^{1,0}\subset \mathcal{H}^{1}$ are the Hodge bundles for the weight $1$ variation of Hodge structure on $H^{1}_\mathbb{Z}$.

The Griffiths infinitesimal invariant $\delta\nu_D(m)$ at the point $m\in M$ is defined as follows: using the exact sequence (\ref{eqexcatpourJ}), one can choose local (for the Euclidean topology)  lifts $$\tilde{\nu}\in\Gamma(\mathcal{H}^1)$$ of $\nu_D$.
Such local lifts are unique up to the addition of  sections $\nu_\mathbb{Z}$ of $H^1_\mathbb{Z}$ and $\nu^{1,0}$ of $\mathcal{H}^{1,0}$. It follows that $\nabla\tilde{\nu}\in \Gamma(\mathcal{H}^1\otimes \Omega_M)$ is defined up to adding a section $\nabla\nu^{1,0}\in \Gamma(\mathcal{H}^1\otimes \Omega_M)$, and thus  the $(0,1)$-projection $(\nabla\tilde{\nu})^{0,1}\in \Gamma(\mathcal{H}^{0,1}\otimes \Omega_M)$ of $\nabla\tilde{\nu}$ is defined modulo
$\overline{\nabla}\mathcal{H}^{1,0}$, where
$$\overline{\nabla}:\mathcal{H}^{1,0}\rightarrow \mathcal{H}^{0,1}\otimes \Omega$$
is the infinitesimal variation of Hodge structures of our family.
As $\overline{\nabla}$ is $\mathcal{O}_M$-linear, the quantity
\begin{eqnarray}\label{eqinfinvt} \delta\nu_D(m):=(\nabla\tilde\nu)^{0,1}_m\in (\mathcal{H}^{0,1}_m\otimes\Omega_{M,m})/\overline{\nabla}(\mathcal{H}^{1,0}_m)
\end{eqnarray}
is thus  well defined at any point $m\in M$. It computes part of  the derivative of the normal function $\nu_D$ and depends only on the data of $\nu_D$ at first order at $m$.
We will use the following
\begin{lemm}\label{lepourlatorsion} Let $\pi:J\rightarrow M$ be a family of abelian varieties of dimension $g$ with corresponding  sheaf  $\mathcal{J}^\vee$ as above and let  $\nu\in \Gamma(M,\mathcal{J}^\vee)$ be a   normal function. Assume that $\nu$ has trivial infinitesimal invariant at the general point  of $M$, and that

(i) The Griffiths differential $\overline{\nabla}:\mathcal{H}^{1,0}\rightarrow  \mathcal{H}^{0,1}\otimes \Omega_M$ is injective (that is, pointwise injective at the general point of $M$).

(ii) The second Griffiths differential $\overline{\nabla}_1:\mathcal{H}^{1,0}\otimes\Omega_M\rightarrow  \mathcal{H}^{0,1}\otimes \Omega_M^2$ is injective (that is, pointwise injective at the general point of $M$).

(iii) The monodromy  group acting on the degree $1$ cohomology of the fibers of the family  $J\rightarrow M$ has finite index in ${\rm Sp}(2g)$.

Then $\nu$ is a   torsion section of $\mathcal{J}^\vee$.
\end{lemm}
In (ii), the map $\overline{\nabla}_1$ is defined by the formula $\overline{\nabla}_1(\omega\otimes \alpha)=\overline{\nabla}\omega\wedge\alpha$.
Note that Condition (ii) implies Condition (i), but we will need the two injectivity statements in the proof.
\begin{rema}{\rm Condition (ii) is actually essential. Consider the case of a $1$-dimensional basis $M$ and assume that (i) is satisfied. Then  $\overline{\nabla}:\mathcal{H}^{1,0}\rightarrow  \mathcal{H}^{0,1}\otimes \Omega_M$ is an isomorphism  at the general point of $M$. Thus the space of  infinitesimal invariants is trivial in this case, so any normal function has trivial infinitesimal invariant at the general point, while of course there always exist nontorsion normal functions, maybe after making a base change.
}
\end{rema}
\begin{proof}[Proof of Lemma \ref{lepourlatorsion}] The argument   is classical and appears in the proof of \cite[Theorem 7.22]{voisinbook} in a slightly different geometric context. We reproduce it here for completeness. Note that it suffices to prove the conclusion after restriction to a dense Zariski open set of $M$, which  does not change  assumption (iii).
Thus we can assume that  $\delta\nu=0$ at any  point of $M$.  This implies  that any  local lift $\tilde{\nu}_0$ of $\nu$ satisfies $(\nabla\tilde{\nu}_0)^{0,1}=(\nabla \eta)^{0,1}$ in $\Gamma(\mathcal{H}^{0,1})$ for some local section $\eta\in \Gamma(\mathcal{H}^{1,0})$, so that the local lift $\tilde{\nu}:=\tilde{\nu}_0-\eta$ of $\nu$ has the property that
\begin{eqnarray}\label{eqpournabaltikd} \nabla{\tilde{\nu}}\in \Gamma(\mathcal{H}^{1,0}\otimes \Omega_M).\end{eqnarray}
We observe now that, since $\nabla\circ\nabla=0$, we have
\begin{eqnarray}\label{eqpournurondnu}
\nabla(\nabla\tilde{\nu})=0.
\end{eqnarray} Then by taking the $(0,1)$-part in (\ref{eqpournurondnu}) and using (\ref{eqpournabaltikd}), we get that
$\overline{\nabla}_1(\nabla\tilde{\nu})=0$. Thus  condition (ii) implies that
we have in fact
\begin{eqnarray}\label{eqpournabaltikdzero} \nabla{\tilde{\nu}}=0,\end{eqnarray}
that is, $\tilde{\nu}$ is a horizontal section of $\mathcal{H}^{1}$.
 Any other local horizontal  local lift $\tilde{\nu}'$ of $\nu$ is of the form
 $$\tilde{\nu}'=\tilde{\nu}+\eta_\mathbb{Z}+\eta^{1,0}$$
  for some local sections $\eta_\mathbb{Z}\in \Gamma(H^1_\mathbb{Z})$ and $\eta^{1,0}\in \Gamma(\mathcal{H}^{1,0})$. We then have
   \begin{eqnarray}\label{eqflateta} \nabla(\eta^{1,0})=0,\end{eqnarray} as  $\tilde{\nu}', \tilde{\nu}$ and $\eta_\mathbb{Z}$ are horizontal. By Condition (i),  (\ref{eqflateta}) implies that $\eta_F$  vanishes identically. It follows that the horizontal local lifts $\tilde{\nu}$ of $\nu$ are unique up to sections of $H^1_\mathbb{Z}$.
We thus proved that  the global section
$\nu\in \Gamma(M,\mathcal{J})$ is induced by a global section $$\overline{\tilde{\nu}}\in \Gamma(M, H^1_\mathbb{C}/H^1_\mathbb{Z}).$$ If $\rho: \pi_1(M,m_0)\rightarrow {\rm Aut}(H^1(J_{m_0},\mathbb{Z}))$ denotes the monodromy representation of our family, such a global section $\overline{\tilde{\nu}}$ gives an element of
$ H^1(J_{m_0},\mathbb{C}/\mathbb{Z})$ which is invariant under $\rho$, or equivalently, an element $\tilde{\nu}_{m_0}\in H^1(J_{m_0},\mathbb{C})$ satisfying $\rho_\gamma(\tilde{\nu}_{m_0})-\tilde{\nu}_{m_0}\in H^1(J_{m_0},\mathbb{Z})$ for any $\gamma\in \pi_1(M,m_0)$. As the image of the monodromy representation contains a finite index subgroup of ${\rm Sp}(H^1(J_{m_0},\mathbb{Z}))$, this last condition is satisfied only when $\tilde{\nu}_{m_0}\in H^1(J_{m_0},\mathbb{Q})$, hence $\overline{\tilde{\nu}}$, and a fortiori $\nu$, is of torsion.
\end{proof}

In \cite{voisinincras}, \cite{voisinICM94}, the following  cohomological  interpretation of $\delta\nu$ is given. Writing the exact
sequence
\begin{eqnarray}\label{eqexactkodspen} 0\rightarrow \pi^*\Omega_M\rightarrow \Omega_J\rightarrow \Omega_{J/M}\rightarrow 0,\end{eqnarray}
restricting it to the fiber $J_m$,
and using the fact that the connecting map of the associated long exact sequence is the map $\overline{\nabla}$ (a fundamental result due to Griffiths, see \cite[Section 5.1.2]{voisinbook}),  one gets a canonical isomorphism
$$H^1(J_m,\Omega_{J\mid J_m})^0\cong (\mathcal{H}^{0,1}_m\otimes\Omega_{M,m})/\overline{\nabla}(\mathcal{H}^{1,0}_m),$$
where on the left  $H^1(J_m,\Omega_{J\mid J_m})^0\subset H^1(J_m,\Omega_{J\mid J_m})$ is defined as the kernel of the restriction map
$$H^1(J_m,\Omega_{J\mid J_m})\rightarrow H^1(J_m,\Omega_{J_m}).$$

\begin{theo}\cite{voisinincras} Let  $D$ be a divisor on $J$ which is homologous to $0$ along the fibers of $\pi$, and $\nu_D$ be the associated normal function. One has \begin{eqnarray}\label{eqpourintcohdedeltanu}
\delta\nu_D(m)=[D]^{1,1}_{\mid J_m}\,\,{\rm in}\,\, H^1(J_m,\Omega_{J\mid J_m})^0,
\end{eqnarray}
where  $[D]^{1,1}\in H^1(J,\Omega_J)$ is the Dolbeault cycle class of $D$, and $[D]^{1,1}_{\mid J_m}$ is its restriction to $J_m$.
\end{theo}

This result allowed us   in \cite{voisinsurfaces}, \cite{voisinICM94} to  define more generally  infinitesimal invariants
$$\delta Z (m)\in H^k(Y_m,\Omega_{Y\mid Y_m}^k)$$
at any point $m\in M$, for higher codimensional cycles
$$Z\subset Y$$
in a smooth  variety $Y$ admitting a smooth fibration $\pi:Y\rightarrow M$  over a base $M$. These infinitesimal invariants are obtained from
the Dolbeault class $[Z]^{k,k}\in H^k(Y,\Omega_Y^k)$, which provides, by restriction to the fibers $Y_m$, the invariant
$$\delta Z (m):=[Z]^{k,k}_{\mid Y_m}\in H^k(Y_m,\Omega_{Y\mid Y_m}^k).$$  In \cite[Section 5.2.1]{voisinbook},  the filtration $L$ on
$H^k(Y_m,\Omega_{Y\mid Y_m}^k)$ induced by the  filtration
$$L^p\Omega_Y^k:=\pi^*\Omega_M^p\wedge \Omega_Y^{k-p}$$
on $\Omega_Y^k$ and its restriction to  $\Omega_{Y\mid Y_m}^k$
 is introduced. Note that $Gr^p_L(\Omega_{Y}^k)\cong \pi^*\Omega_{M}\otimes\Omega_{Y/M}^{k-p}$.  This filtration allows to introduce successive infinitesimal invariants living in  the graded pieces of $Gr^p_LH^k(Y_m,\Omega_{Y\mid Y_m}^k)$. The following result is proved in \cite[Proposition 5.9]{voisinbook} concerning the corresponding spectral sequences
 \begin{eqnarray}\label{eqspectseq} E_r^{p,q}\Rightarrow H^{p+q}(Y_m,\Omega_{Y\mid Y_m}^k),\\ \nonumber
 E_r^{p,q}\Rightarrow R^{p+q}\pi_*(\Omega_{Y}^k)  .
 \end{eqnarray}
 \begin{prop}\label{protrivbook}  The $E_1$-terms of (\ref{eqspectseq})
 $$\ldots E_1^{p-1, l-p}\stackrel{d_1}\rightarrow E_1^{p,l-p}\stackrel{d_1}\rightarrow E_1^{p+1,l-p}\ldots$$  identify respectively with  the
 Griffiths complexes
$$\ldots H^{k-p+1,l-1}(Y_m)\otimes \Omega_{M,m}^{p-1} \stackrel{\overline{\nabla}_{p-1}}{\rightarrow }  H^{k-p,l}(Y_m)\otimes \Omega_{M,m}^p\stackrel{\overline{\nabla}_p}{\rightarrow } H^{k-p-1,l+1}(Y_m)\otimes \Omega_{M,m}^{p+1}\ldots,$$
$$ \ldots \mathcal{H}^{k-p+1,l-1}\otimes \Omega_{M}^{p-1} \stackrel{\overline{\nabla}_{p-1}}{\rightarrow }  \mathcal{H}^{k-p,l}\otimes \Omega_{M}^p\stackrel{\overline{\nabla}_p}{\rightarrow } \mathcal{H}^{k-p-1,l+1}\otimes \Omega_{M}^{p+1}\ldots,  $$
where in the second formula $\mathcal{H}^{p,q}:=R^q\pi_*\Omega_{Y/M}^p$ and  we set $\overline{\nabla}_p(\alpha\otimes \omega)=\overline{\nabla}(\alpha)\wedge \omega$ for any
$\alpha\in H^{p,q}(Y_m)$ and $\omega\in \Omega_{M,m}$.
\end{prop}

 It the paper \cite{voisinsurfaces} (see also \cite{voisinICM94}) where infinitesimal invariants are introduced using  this spectral sequence,   it degenerates at $E_2$ for degree reasons, because the infinitesimal variation of Hodge structures considered are that of hypersurfaces, so are supported in only one degree.  We could not find a   reference for its  degeneracy at $E_2$ in general, so we include here a proof for completeness.

 \begin{prop} \label{propspecseq} Let  $\pi: Y\rightarrow M$ be  a  smooth projective morphism, where $M$ is smooth and  quasi-projective. Then the spectral sequences (\ref{eqspectseq}) degenerate at $E_2$.
 \end{prop}

\begin{rema}{\rm  In this paper, we are going to study infinitesimal invariants for cycles on families $\pi: J\rightarrow M$ of abelian varieties. In this case, proving the degeneracy  at $E_1$ of the  spectral sequences (\ref{eqspectseq}) is even easier. Indeed, for any integer  $n>1$, let $\mu_n:J\rightarrow J$ be the multiplication by $m$ map. It is a morphism of fibrations over $M$, and  the  pull-back morphisms
$$ \mu_n^*: \mu_n^*\Omega^k_J\rightarrow \Omega_J^k$$
are compatible with the filtration $L$,  which provide filtered
isomorphisms
$$\mu_n^*: H^k(J_m,\Omega_{J\mid J_m}^k)\rightarrow H^k(J_m,\Omega_{J\mid J_m}^k)$$
$$ \mu_n^*: R^k\pi_*\Omega_{J}^k\rightarrow R^k\pi_*\Omega_{J}^k.$$
The morphisms $\mu_n^*$ induce morphisms of spectral sequences (\ref{eqspectseq}) for $J$, and by Proposition \ref{protrivbook}, these morphisms act as  $n^{k-p+l}Id$ on $E_1^{p,l-p}$. It thus follows that they also act  as  $n^{k-p+l}Id$ on $E_r^{p,l-p}$, for any $r\geq 1$. But then, the differentials
$$d_r:E_r^{p,l-p}\rightarrow E_r^{p+r,l-p-r+1}$$
vanish for $r\geq 2$ since they are compatible with the $\mu_n$-action and $\mu_n^*$ acts by different powers of $n$ on both sides for $r\geq 2$. }
\end{rema}
\begin{proof}[Proof of Proposition \ref{propspecseq}] The reasoning is similar to the one presented above, except that we need a substitute for the $\mu_n^*$. This is done using the relative K\"{u}nneth projectors constructed as follows.
Let $d$ be the relative dimension of $\pi$. Consider $Y_2:=Y\times_MY\stackrel{\pi'}{\rightarrow} M$, with projections $\pi_1,\,\pi_2: Y_2\rightarrow Y$, which are both smooth projective morphisms of relative dimension $d$. The diagonal $\Delta_Y\subset Y_2$ has a cohomology class
$$\delta_Y\in H^{2d}(Y_2,\mathbb{Q})$$
which is algebraic hence extends to  a Hodge class on a given projective compactification $\overline{Y}_2$ of $Y_2$. It induces a relative action
$$\delta_{Y,*}:R^j \pi_{1*}\mathbb{Q}\rightarrow R^j \pi_{2*}\mathbb{Q},$$
$$\delta_{Y,*}(\alpha)=\pi_{2*}(\pi_1^*\alpha\cup \delta_Y)$$
which is nothing but the identity.
The class $\delta_Y$ induces a global section of
\begin{eqnarray}\label{eqtropchaud} R^{2d}\pi'_{*}\mathbb{Q}=\oplus_{j+i=2d}R^j\pi{*}\mathbb{Q}\otimes R^i\pi{*}\mathbb{Q}=\oplus_i{Hom}(R^i\pi{*}\mathbb{Q}, R^i\pi{*}\mathbb{Q}).\end{eqnarray}
Using (\ref{eqtropchaud}), we can define  the global section $\delta_{Y,i}$ of the local system $R^{2d}\pi'_{*}\mathbb{Q}$  as the one which acts as   the identity on  $R^i\pi{*}\mathbb{Q}$ and by $0$ on $R^j\pi{*}\mathbb{Q}$ for $j\not=i$. The classes  $\delta_{Y,i}(m)$  are Hodge classes at any point $m\in M$ which are  not known to be algebraic but by Deligne's
global invariant cycles theorem  \cite{deligne}, see also \cite[Theorem 4.24]{voisinbook}, there exists for any $i$  a Hodge class on
$\overline{Y}_2$, such that its restriction
$\delta_i$ to $Y_2$ induces the section $\delta_{Y,i}\in H^0(M, R^{2d}\pi_{2*}\mathbb{Q})$. The class
$\delta_i\in H^{2d}(Y_2,\mathbb{Q})$ (which is what we call a  relative K\"{u}nneth projector) is thus constructed so that
\begin{eqnarray}\label{eqpouractiondedelatai} \delta_{i*}(\alpha):=\pi_{2*}(\pi_1^*\alpha\cup \delta_i ) :R^j\pi_*\mathbb{Q}\rightarrow R^j\pi_*\mathbb{Q}
\end{eqnarray}
is the identity for $j=i$ and $0$ otherwise.
As $\delta_i$ is the restriction of a Hodge class on $\overline{Y}_2$,  there is a Dolbeault version
$$\delta_i^{d,d}\in H^d(Y_2,\Omega_{Y_2}^d)$$ of $\delta_i$. Let us choose a Dolbeault representative
$\tilde{\delta}_i^{d,d}$ of $\delta_i^{d,d}$, given by   a closed $(d,d)$-form on $Y_2$.

\begin{lemm}  \label{leactionkunnethcompdolb} (i) For  each $k$, the $\overline{\partial}$-closed forms $\tilde{\delta}_i^{d,d}$ act on the  push-forward to $M$ of the  Dolbeault complex $(\mathcal{A}^{0,*}(\Omega_Y^k),\overline{\partial}_{Y})$ of $\Omega_Y^k$,  whose cohomology computes
$R^*\pi_*\Omega_Y^k$. This  action is compatible with the filtration $L$ on $\Omega_Y^k$.

(ii) The induced action of $\tilde{\delta}_i^{d,d}$ on $E_1^{p,l-p}=\mathcal{H}^{k-p,l}\otimes \Omega_{M}^p$ is by the identity if $k-p+l=i$ and by $0$ otherwise.

(iii) Similar statement at any point $m\in M$.
\end{lemm}
\begin{proof} The action of $\tilde{\delta}_i^{d,d}$ on the relative Dolbeault complex of $\Omega_Y$  is given by
\begin{eqnarray}\label{eqformula20juin} \tilde{\delta}_{i*}(\alpha)=\pi_{2*}(\pi_1^*\alpha\cup \tilde{\delta}_{i}),\end{eqnarray}
where $\pi_{2*}: A^{p,q}(Y_2/M)\rightarrow A^{p-d,q-d}(Y/M)$ is given by integration along the fibers. As the three operations $\pi_1^*, \,\cup \tilde{\delta}_{i}$ and $\pi_{2*}$ commute with the  Dolbeault operator $\overline{\partial}$ and preserve the $L$-filtration, so does $\tilde{\delta}_{i*}$.
The $\tilde{\delta}_{i*}$ thus provide morphisms of filtered complexes quasi-isomorphic to $\Omega_Y^k$, hence morphisms of the corresponding spectral sequences.  This proves (i). Concerning (ii),  the action of
$\tilde{\delta}_{i*}$ on $E_1^{p,l-p}=\mathcal{H}^{k-p,l}\otimes \Omega_{M}^p$ is compatible with its  action of $\mathcal{H}^{k-p,l}$ (and is the identity on $\Omega_{M}^p$), and the latter is also compatible with the action of $\tilde{\delta}_{i*}$ on $R^{k-p+l}\pi_*\mathbb{C}$ given by the same formula as (\ref{eqformula20juin}), which makes sense since we chose $\tilde{\delta}_{i*}$ to be a closed form. This latter action is the action of $\delta_{i*}$ given by (\ref{eqpouractiondedelatai}), hence it vanishes if
$k-p+l\not=i$, and it is the identity if $k-p+l=i$, which proves (ii).
\end{proof}
By Lemma \ref{leactionkunnethcompdolb}(i), the form $\sum_i\tilde{\delta}_{i}$ acts on the $L$-filtered  relative Dolbeault complex of $\Omega_Y^k$, hence on the whole spectral sequences (\ref{eqspectseq}). Furthermore by Lemma \ref{leactionkunnethcompdolb}(ii), it acts by the identity on the $E_1$-term, hence on all $E_r$-terms. In order to prove that $d_r=0$ for $r\geq 2$, it thus suffices to prove that
$d_r\circ \tilde{\delta}_{i*}=0$ for all $r\geq 2$. This is now obvious for degree reason by Lemma \ref{leactionkunnethcompdolb}(ii).
\end{proof}
 A generalization of Mumford's theorem (see \cite{voisinsurfaces}, \cite{voisinchow}) is as follows.
\begin{theo}\label{theomumfordvar}  Let $\pi:Y\rightarrow M$ be a  smooth  projective morphism, and let  $Z\in {\rm CH}^k(Y)$ satisfy the condition  that for any fiber $m\in M$, $Z_{\mid Y_m}=0$ in ${\rm CH}^k(Y_m)$. Then, for a general point  $m\in M$,
$\delta Z(m)=0$ in  $H^k(Y_m,\Omega_{Y\mid Y_m}^k)$.
\end{theo}
Indeed, under these assumptions, a multiple $NZ$ is rationally equivalent to $0$ on $Y_{U}$, for some dense Zariski open set $U\subset M$ (see \cite{blochsri},  \cite{voisinchow}). A fortiori, one has $[Z]^{k,k}_{\mid Y_{U}}=0$ in $H^k(Y_{U},\Omega^k_{Y_{U}})$ and so $\delta Z(m)=0$ for any $m\in U$, by restriction to $Y_m$. Note that the vanishing of $\delta Z(m)$ may not hold everywhere, due to the fact that
$R^k\pi_*\Omega_Y^k$ is not  torsion free in general.

The proofs of Theorems \ref{theoabelian} and \ref{theoquestionpirola} rely on Theorem \ref{theomumfordvar} applied to the case of codimension $2$ cycles on an abelian fibration $J\rightarrow M$.  More precisely, the cycles $Z$ that we are going to consider are of the form $Z=D^2$, for some divisor $D$ on $J$. The compatibility of the cycle class with cup-product and intersection product gives the equality
\begin{eqnarray}\label{eqcarrecarrecycle} [Z]^{2,2}=([D]^{1,1})^2\,\,{\rm in}\,\,H^2(J,\Omega_J^2).
\end{eqnarray}
Restricting this equalities along the fibers of $\pi$, one gets that $\delta Z$ is computed  explicitly as the square of $\delta D$, although the  square map

$${\rm Sym}^2 H^1(J_m,\Omega_{ J\mid J_m})\rightarrow H^2(J_m, \Omega_{J\mid J_m}^2)$$
is not so easy to analyze. This analysis was started in \cite{verni} and our results rely on the analysis made in {\it loc. cit.}. The paper \cite{verni} also computes the infinitesimal invariant  for the Griffiths-Pirola normal function $\gamma_P$ (defined over a double cover $\widetilde{\mathcal{M}}_4$ of $\mathcal{M}_4$), and for its square. Let $\widetilde{\mathcal{C}}\rightarrow \widetilde{\mathcal{M}}_4$ be the pull-back of the universal curve of genus $4$ as in the introduction, and let $\widetilde{\mathcal{C}}^{(2)/\widetilde{\mathcal{M}}_4}\rightarrow \widetilde{\mathcal{M}}_4$ be its relative second symmetric product. The normal function $\gamma_P$  provides as well a divisor $D_{\gamma_P}$ on $\widetilde{\mathcal{C}}^{(2)}$, which is cohomologous to $0$ along fibers.

 The following results proved in \cite{verni} will be  used in Section \ref{secpirola}.  Theorem  \ref{theovernicubic} strengthens an earlier result by Griffiths \cite{griffiths}, and  fully computes the infinitesimal invariant  $\delta\gamma_{P,[C]}$ for  a general curve $C$ of genus $4$. Let $C$ be  a general curve of genus $4$. Let $$V:=H^{1,0}(C), \,\,V_2:=H^0(C,K_C^{\otimes2})={\rm Sym}^2V/q,$$ where $q\in {\rm Sym}^2V$ is   nondegenerate and defines the unique quadric  $Q$ containing $C$ in its canonical embedding. The space of infinitesimal invariants  at $[C]$ for normal functions $\nu\in \Gamma(\mathcal{J}^\vee)$ defined on a neighborhood of $[C]\in\mathcal{M}_4$
is
$$I_{1,V}={\rm Hom}(V,V_2)/V$$
where the inclusion  $V\hookrightarrow   {\rm Hom}(V,V_2)$  maps $v\in V$ to the multiplication map  by $v$. The space $I_{1,V}$ is a representation of the group ${\rm Aut}(V,q)$, and, as such,  splits naturally as the direct sum of $H^0(Q,\mathcal{O}_Q(3))$ and
$H^0(Q,T_Q(1))$.

\begin{theo}\label{theovernicubic} \cite[Theorem 1.2]{verni} The infinitesimal invariant $\delta\gamma_P$ at $[C]$ (which is well-defined up to sign) belongs to  $H^0(Q,\mathcal{O}_Q(3))\subset I_{1,V}$. It is  a nonzero multiple of  the cubic equation defining $C\subset Q$.
\end{theo}

\begin{theo}\label{theoverni} \cite[Theorem 1.3]{verni}  One  has $(\delta\gamma_P)^2=\delta (D_{\gamma_P}^2)\not=0$ in $H^2(J_m,\Omega_{J\mid J_m}^2)$ at a general point $m\in \widetilde{\mathcal{M}}_4$, hence for a very general point $m\in \widetilde{\mathcal{M}}_4$, one has $D_{\gamma_P}^2\not=0$ in ${\rm CH}_0(C^{(2)})$.
\end{theo}
The second statement in Theorem \ref{theoverni} follows from the first by Theorem \ref{theomumfordvar}.

\begin{rema}{\rm Theorem \ref{theoverni} is also implied by Theorem \ref{theoabelian}(ii). However the proof in \cite{verni} is much  more direct, because it directly proves by an explicit computation
 the nonvanishing of the infinitesimal invariant $\delta (D_{\gamma_P}^2)$, hence uses  none of the ingredients of Section \ref{secsec3ou41506}, apart of course for the general theory of infinitesimal invariants.}
 \end{rema}
 \begin{rema}{\rm  The spaces of  infinitesimal invariants for codimension $2$ cycles on $J(C)$ and on $C^{(2)}$ coincide. So the (non)vanishing properties of the infinitesimal invariants for $C^{(2)}$ and for $J(C)$ are equivalent, which is coherent with Lemma \ref{lenouveau2804}. However, it is important to also  consider in the course of the paper cycles on  $C^{(2)}$  because the Chow group of codimension $2$ cycles on $C_\eta^{(2)}$, where $C_\eta$ is the generic curve of genus $4$ is well understood, see for example \cite{voisinhodgebloc}. }
 \end{rema}
\section{Proof of Theorem \ref{theoabelian}\label{secsec3ou41506}}
\subsection{The equation $\overline{M}_2(\overline{\phi})=0$ \label{secthsquare}}
Let $V$ be a $g$-dimensional vector space, with $g\geq4$. We denote by $W_2$ the vector space ${\rm Sym}^2V$. If $g=4$, given a nondegenerate quadratic form $q$ on $V$,  we denote  by $V_2$ the  vector space ${\rm Sym}^2V/q$. If $A$ is a principally polarized  abelian variety of dimension $g$, and $V=H^{1,0}(A)$, $W_2$ is the cotangent space at $[A]$ to the moduli stack $\mathcal{A}_{g}$ (or to the base of the Kuranishi family) of polarized deformations of $A$. If $A=J(C)$, where $C$ is a non-hyperelliptic curve of genus $4$, and $q$ is the unique quadratic form vanishing on $C$ in its canonical embedding, $q$ is nondegenerate if $C$ is general, and  $V_2=H^0(C,2K_C)$ is the cotangent space at $[C]$ to the moduli stack $\mathcal{M}_4$ of  deformations of $C$.

Let \begin{eqnarray}\label{eqI1V} I_{1,W}:={\rm Hom}(V,W_2)/V,\,\,{\rm resp.}\,\,I_{1,V}:={\rm Hom}(V,V_2)/V,
\end{eqnarray}
where, in both cases, $V$ is naturally embedded in ${\rm Hom}(V,W_2)$, resp.  ${\rm Hom}(V,V_2)$, via the multiplication maps.
Let $J_g\rightarrow \mathcal{A}_g$ be the universal principally polarized abelian variety of dimension $g$  and let  $U$ be an (analytic or \'{e}tale) open set  of  $\mathcal{A}_g$. Let  $D$ be a divisor on  $J_{g,U}$, which is homologous to $0$ on fibers of $\pi$.  As we saw in Section \ref{secreview}, using the exact sequence
 \begin{eqnarray}\label{eqexconormal2704} 0\rightarrow \Omega_{U,m}\otimes \mathcal{O}_{J_{g,m}}\rightarrow \Omega_{J_g\mid J_{g,m}}\rightarrow \Omega_{ J_{g,m}}\rightarrow 0
 \end{eqnarray}
 at any point  $m\in U$, the infinitesimal invariant $\delta D(m)=\delta \nu_{D,m}$ at   $m\in U$ belongs to
$$H^1(J_{g,m},\Omega_{J_g\mid J_{g,m}})^0=(H^{0,1}({J}_{g,m})\otimes \Omega_{U,m})/\overline{\nabla}(H^{1,0}(J_{g,m})),$$
 and
 the right hand side is naturally isomorphic to $I_{1,W}$, with $V=H^{1,0}(J_{g,m})$.

 Similarly, if $\pi: J\rightarrow \mathcal{M}_4$ is the universal Jacobian fibration of the universal family of curves $\mathcal{C}\rightarrow \mathcal{M}_4$ of genus $4$, and $V_2$ is as above, the space
 $$I_{1,V}\cong H^1(\Omega_{J\mid J_{m}})^0\cong H^1(\Omega_{\mathcal{C}\mid C})^0,$$
  where $C\cong \mathcal{C}_m$, is  the space of infinitesimal invariants at $m=[C]$  for normal functions $\nu_D\in \Gamma( \mathcal{J}^\vee)$ defined on open (\'{e}tale or analytic) sets of $\mathcal{M}_4$.

For any $\phi,\,\psi\in V^*\otimes W_2$, we denote by $\phi\cdot\psi\in {\rm Sym}^2(V^*\otimes W_2)$ the image of $\phi\otimes \psi\in (V^*\otimes W_2)^{\otimes2}$. Similarly, recalling that $V$ embeds naturally in ${\rm Hom}(V,W_2)=V^*\otimes W_2$, we denote by
$V\cdot (V^*\otimes W_2)\subset {\rm Sym}^2(V^*\otimes W_2)$ the subspace generated by elements $\phi\cdot \psi$,  where $\phi\in V$. The quotient
$$ {\rm Sym}^2(V^*\otimes W_2)/ V\cdot (V^*\otimes W_2)$$
is naturally isomorphic to ${\rm Sym}^2((V^*\otimes W_2)/V)={\rm Sym}^2 I_{1,W}$.

The $2\times 2$-minors map
$$V^*\otimes W_2\ni\phi\mapsto M_{2,W}(\phi):=\wedge^2\phi\in \bigwedge^2V^*\otimes \bigwedge^2W_2$$ extends by linearity to  a morphism
$$M_{2,W}: {\rm Sym}^2(V^*\otimes W_2)\rightarrow \bigwedge^2V^*\otimes \bigwedge^2W_2.$$

Let \begin{eqnarray}\label{eqI2W}I_{2,W}:=(\bigwedge^2V^*\otimes \bigwedge^2W_2)/M_{2,W}(V\cdot (V^*\otimes W_2)).
\end{eqnarray}
 We denote
by
\begin{eqnarray}\label{eqM2bar}\overline{M}_{2,W}: {\rm Sym}^2((V^*\otimes W_2)/V)={\rm Sym}^2 I_{1,W}\rightarrow I_{2,W}\end{eqnarray}
the natural factorization.

In the second situation, where ${\rm dim}\,V=4$ and we work with $V_2=W_2/q$, we will use the similarly defined notation $M_{2,V}: {\rm Sym}^2(V^*\otimes V_2)\rightarrow \bigwedge^2V^*\otimes \bigwedge^2V_2$, and
\begin{eqnarray}\label{eqI2V2704}I_{2,V}:=(\bigwedge^2V^*\otimes \bigwedge^2V_2)/M_{2,V}(V\cdot (V^*\otimes V_2)).
\end{eqnarray}
As before, we also have the induced morphism

\begin{eqnarray}\label{eqM2barpourV}\overline{M}_{2,V}: {\rm Sym}^2((V^*\otimes V_2)/V)={\rm Sym}^2 I_{1,V}\rightarrow I_{2,V}.\end{eqnarray}

Coming back to the first situation, we have as explained in Section \ref{secreview} a filtration  $L^\bullet $ on
$H^2(\Omega_{J_{g}\mid J_{g,m}}^2)$. Given a divisor $D$ homologous to $0$ on fibers of $\pi$ on $J_{g}$, the infinitesimal invariant $$\delta (D^2)=(\delta D)^2$$ clearly  belongs to the deepest part
$$L^2H^2(\Omega_{J_{g}\mid J_{g,m}}^2)\subset H^2(\Omega_{J_{g}\mid J_{g,m}}^2)$$ for this filtration.
This deepest part  is  computed using Propositions \ref{protrivbook} and \ref{propspecseq} as
\begin{eqnarray}\label{eqpourinf22704} L^2H^2(\Omega_{J_{g}\mid J_{g,m}}^2)\cong (H^{0,2}(J_{g,m})\otimes \Omega_{U,m}^2)/\overline{\nabla}_1(H^{1,1}(J_{g,m})\otimes \Omega_{U,m}),\end{eqnarray}
where $\overline{\nabla}_1:  H^{1,1}(J_{g,m})\otimes \Omega_{U,m}\rightarrow H^{0,2}(J_{g,m})\otimes \Omega_{U,m}$ is given by
$$\overline{\nabla}_1(\omega\otimes \alpha)=\overline{\nabla}(\omega)\wedge\alpha.$$
Comparing  (\ref{eqpourinf22704}) and (\ref{eqI2W}), one gets (see \cite{verni} for more detail)  that the  space $L^2H^2(\Omega_{J_{g}\mid J_{g,m}}^2)$ is naturally  isomorphic to the space $I_{2,W}$.
It is also proved in   \cite{verni} that the linear map
$$ {\rm Sym}^2H^1(\Omega_{J_{g}\mid J_{g,m}})^0\rightarrow L^2H^2(\Omega_{J_{g}\mid J_{g,m}}^2)$$
induced by
the square map
$$ H^1(\Omega_{J_{g}\mid J_{g,m}})^0= L^1H^1(\Omega_{J_{g}\mid J_{g,m}})\ni \delta D\mapsto (\delta D)^2\in L^2H^2(\Omega_{J_{g}\mid J_{g,m}}^2)$$
is nothing but the product $\overline{M}_{2,W}$ of (\ref{eqM2bar}).
Similarly, in the second situation where we work over open sets $U$ of  $\mathcal{M}_4$, with tangent space $V_2$, and $J\rightarrow \mathcal{M}_4$ is the universal Jacobian,  the space $L^2H^2(\Omega_{J\mid J_{m}}^2)$ is naturally isomorphic to the space  $I_{2,V}$ of (\ref{eqI2V2704})  and it is proved in \cite[Proposition 5.1]{verni} that the linear map
$$ {\rm Sym}^2H^1(\Omega_{J\mid J_{m}})^0\rightarrow  L^2H^2(\Omega_{J\mid J_m}^2)
$$
induced by the square map
$$ H^1(\Omega_{J\mid J_{m}})^0= L^1H^1(\Omega_{J\mid J_{m}})\ni \delta D\mapsto (\delta D)^2\in L^2H^2(\Omega_{J\mid J_m}^2)$$
is nothing but the product $\overline{M}_{2,V}$ of (\ref{eqM2barpourV}).

\begin{prop}\label{propcarrenul}  The  notation being as above, the following hold:

(i)  Assume $g\geq 4$. If
$ \overline{\phi}\in I_{1,W}=(V^*\otimes W_2)/V$ has the property that
$$\overline{M}_{2,W}(\overline{\phi})=0\,\,{\rm in}\,\, (\bigwedge^2 V^*\otimes \bigwedge^2W_2)/M_2(V\cdot (V^*\otimes W_2))=I_{2,W},$$
there exists a unique $\phi \in V^*\otimes W_2$ mapping to $\overline{\phi}$, such that $$M_2(\phi)=0\,\,{\rm  in }\,\,\bigwedge^2 V^*\otimes \bigwedge^2W_2,$$ that is, $\phi$ has rank $\leq 1$.

(ii) Assume $g=4$. If
$ \overline{\phi}\in I_{1,V}=(V^*\otimes V_2)/V$ has the property that
$$\overline{M}_{2,V}(\overline{\phi})=0\,\,{\rm in}\,\, (\bigwedge^2 V^*\otimes \bigwedge^2V_2)/M_2(V\cdot (V^*\otimes V_2))=I_{2,V},$$
there exists a unique $\phi \in V^*\otimes V_2$ mapping to $\overline{\phi}$, such that
$$M_2(\phi)=0\,\,{\rm in }\,\,\bigwedge^2 V^*\otimes \bigwedge^2V_2,$$ that is, $\phi$ has rank $\leq 1$.
\end{prop}
In both cases, the uniqueness follows from the assumption $g\geq 4$ (in fact $g\geq 3$ would suffice here). Indeed, assume there are two lifts $\phi,\,\phi'\in V^*\otimes W_2$, resp. $\phi,\,\phi'\in V^*\otimes V_2$ of $\overline{\phi}\in I_{1,W}$ (resp.  $\overline{\phi}\in I_{1,V}$) that have rank $\leq 1$. Then the difference $\phi-\phi'$ has rank $\leq 2$, and belongs to $V\subset {\rm Hom}(V,W_2)$ (resp.   to $V\subset {\rm Hom}(V,V_2)$). This contradicts the fact that, for any $0\not=v\in V$, the   map $\mu_v: V\rightarrow W_2$, (resp.  $\mu_v: V\rightarrow V_2$) of multiplication by $v$ is injective, hence has rank $g$ (resp. rank $4$).

The unicity being proved, we  concentrate below on the existence of the rank $1$ lifts.
\begin{proof}[Proof of Proposition \ref{propcarrenul}(i)]
For any  $2$-dimensional vector  subspace $S\subset V$, let $E:=V/S$.  We have two  restriction maps $V^*\rightarrow S^*$ and $W_2\rightarrow {\rm Sym}^2E$, which together  induce
$$r_S: V^*\otimes W_2\rightarrow S^*\otimes {\rm Sym}^2E.$$
We use the following observation made in \cite[Lemma 6.1]{verni}.
\begin{lemm}\label{leverni} For any $v\in V$, with induced  morphism
$$\mu_v\in {\rm Hom}(V,W_2)=V^*\otimes  W_2$$
of multiplication by $v$,
we have $r_S(\mu_V)=0$ in $S^*\otimes {\rm Sym}^2E$.
\end{lemm}
From the restriction maps $V^*\rightarrow S^*$ and $W_2\rightarrow {\rm Sym}^2E$, we also deduce a restriction morphism $$r'_S: \bigwedge^2V^*\otimes \bigwedge^2W_2\rightarrow \bigwedge^2S^*\otimes \bigwedge^2({\rm Sym}^2E),$$
which is compatible with $r_S$ and the $2\times 2$-minors maps on both sides.   We then  get  from Lemma \ref{leverni} that $r'_S$ vanishes on $M_{2,W}(V\cdot (V^*\otimes W_2))$. Hence, using formula (\ref{eqI2W}) defining $I_{2,W}$, we  deduce  that, for any vector subspace $S\subset V$ of dimension $2$,  there is   a linear map
\begin{eqnarray}\label{eqnewdu26mai} r_{S,2}:I_{2,W}\rightarrow {\rm Hom}(\bigwedge^2S,\bigwedge^2({\rm Sym}^2E))\end{eqnarray}
obtained by  factorization of $r'_S$,
such that, for any $\overline{\phi}\in V^*\otimes W_2/V$, we have
\begin{eqnarray}\label{eqrestcarG} r_{S,2}(\overline{M}_2(\overline{\phi}))=M_{2,S}(r_S(\overline{\phi})),\end{eqnarray}
where the map $M_{2,S}$ on the right is the $2 \times 2$-minor map for ${\rm Hom}(S,{\rm Sym}^2E)$.

Let $G=G(2,V)$ be the Grassmannian of $2$-dimensional vector subspaces of $V$. Let
$\mathcal S\subset V\otimes \mathcal{O}_G$ be the tautological rank $2$ vector subbundle, and $\mathcal{E}:=V\otimes \mathcal{O}_G/\mathcal{S}$ the rank $g-2$ tautological quotient bundle.
Lemma \ref{leverni} provides a linear map

\begin{eqnarray}\label{eqrone}r_1: (V^*\otimes W_2)/V\rightarrow  H^0(G,\mathcal{H}om(\mathcal{S},{\rm Sym}^2\mathcal{E})).\end{eqnarray}

We also have the $2\times 2$-minors quadratic map
\begin{eqnarray}\label{eqmtwo}m_2: \mathcal{H}om(\mathcal{S},{\rm Sym}^2\mathcal{E}))\rightarrow  \mathcal{H}om(\bigwedge^2\mathcal{S},\bigwedge^2({\rm Sym}^2\mathcal{E}))
\\
\nonumber
m_2(\psi):=\wedge^2\psi.\end{eqnarray}
We will also denote  by $m_2$ the map induced by (\ref{eqmtwo}) at the level of global sections. By (\ref{eqnewdu26mai}),  we get a linear map
\begin{eqnarray}\label{eqrtwo}r_2: I_{2,W}\rightarrow H^0(G,\mathcal{H}om(\bigwedge^2\mathcal{S},\bigwedge^2({\rm Sym}^2\mathcal{E})))\end{eqnarray}
such that, for any $\overline{\phi}\in (V^*\otimes W_2)/V$, we have as in (\ref{eqrestcarG})
\begin{eqnarray}\label{eqfcatorization} r_2(\overline{M}_2(\overline{\phi}))=m_2(r_1(\overline{\phi}))
\,\,{\rm in}\,\, H^0(G,\mathcal{H}om(\bigwedge^2\mathcal{S}, \bigwedge^2({\rm Sym}^2\mathcal{E}))).\end{eqnarray}

We now prove  two lemmas.
\begin{lemm} \label{le1pourcar0} For $g\geq 4$, the morphism $r_1$ of (\ref{eqrone}) is injective.
\end{lemm}
\begin{proof} Let $\phi\in V^*\otimes W_2$ such that $r_1(\overline{\phi})=0$ in $H^0(G,\mathcal{H}om(\mathcal{S},{\rm Sym}^2\mathcal{E}))$. We note that, given a quotient
$E=V/S$, the kernel of the natural map ${\rm Sym}^2V\rightarrow {\rm Sym}^2E$  maps naturally to $S\otimes E$, with kernel isomorphic to ${\rm Sym}^2S$. By our assumption, $\phi$ thus induces a section $$s\in H^0(G,{ Hom}(\mathcal{S},\mathcal{S}\otimes \mathcal{E})).$$ We now write the exact sequence
$$0\rightarrow \mathcal{S}\rightarrow V\otimes \mathcal{O}_G\rightarrow \mathcal{E}\rightarrow 0$$
on $G$, tensor it  by $\mathcal{E}nd(\mathcal{S})$, and get by taking sections the long exact sequence
$$ \ldots \rightarrow H^0(G, \mathcal{E}nd(\mathcal{S})\otimes V)\rightarrow H^0(G,{ \mathcal{H}om}(\mathcal{S},\mathcal{S}\otimes \mathcal{E}))\rightarrow H^1(G, \mathcal{E}nd(\mathcal{S})\otimes \mathcal{S})\rightarrow\ldots.$$
The vanishing of $H^1(G, \mathcal{E}nd(\mathcal{S})\otimes \mathcal{S})$ and the fact that $H^0(G,\mathcal{E}nd(\mathcal{S}))=\mathbb{C}{\rm Id}$ (thanks to the assumption $g\geq 4$) thus imply that $s$ comes from an element $v$ of $V$. Thus the image of $\phi-\mu_v$ in $H^0(G, { \mathcal{H}om}(\mathcal{S},{\rm Sym}^2V\otimes \mathcal{O}_G))$ belongs to $H^0(G,{\mathcal{H}om}(\mathcal{S},{\rm Sym}^2\mathcal{S}))$, which is $0$ because $g\geq3$. This  implies that $\phi=\mu_v$.
\end{proof}

\begin{lemm} \label{le2pourcar0} If $0\not=\alpha\in H^0(G,\mathcal{H}om(\mathcal{S}, {\rm Sym}^2\mathcal{E}))$ satisfies
\begin{eqnarray}\label{eqcar0surG} m_2(\alpha)=0\,\,{\rm in}\,\, H^0(G,\mathcal{H}om(\bigwedge^2{S}, \bigwedge^2({\rm Sym}^2\mathcal{E}))),\end{eqnarray}
then there exists a rank $1$ element
$\psi\in V^*\otimes W_2$ whose image in $H^0(G,\mathcal{H}om(\mathcal{S},{\rm Sym}^2\mathcal{E}))$ equals $\alpha$.
\end{lemm}
\begin{proof} The condition $m_2(\alpha)=0$ means that $\alpha $ has everywhere rank $\leq 1$. Thus, if $\alpha\not=0$, there is a line bundle $\mathcal{L}$ on $G$ and two nonzero  morphisms
\begin{eqnarray}\label{eqbetaeti} \beta:\mathcal{S}\rightarrow \mathcal{L},\, i:\mathcal{L}\rightarrow {\rm Sym}^2\mathcal{E},\end{eqnarray}
such that \begin{eqnarray}\label{eqalphafactor} \alpha=i\circ \beta: \mathcal{S}\rightarrow {\rm Sym}^2\mathcal{E}.\end{eqnarray}
 The line bundle $\mathcal{L}\in {\rm Pic}(G)$ is a power of the Pl\"{u}cker line bundle. It cannot be strictly positive, because the existence of the morphism $i\not=0$ shows that  $H^0(G,{\rm Sym}^2\mathcal{E}\otimes \mathcal{L}^{-1})\not=0$. The line bundle $\mathcal{L}$ cannot be strictly negative because $H^0(G,\mathcal{S}^*\otimes \mathcal{L})\not=0$. It follows that $\mathcal{L}=\mathcal{O}_G$. Using (\ref{eqbetaeti}) with $\mathcal{L}$ trivial,  (\ref{eqalphafactor}) gives that the morphism $\alpha\in H^0(G,\mathcal{S}^*\otimes {\rm Sym}^2\mathcal{E})$  can be written as $u\otimes v$, where $u\in H^0(G,\mathcal{S}^*)=V$, and $v\in H^0(G,{\rm Sym}^2\mathcal{E})={\rm Sym}^2H^0(G,\mathcal{E})=W_2$. Thus Lemma \ref{le2pourcar0} is proved.
\end{proof}
Using formula (\ref{eqrestcarG}), Lemma \ref{le2pourcar0} shows that for $\overline{\phi}\in V^*\otimes W_2/V$, the condition $\overline{M}_{2,W}(\overline{\phi})=0$ implies that there exists a rank $1$ element $\psi\in V^*\otimes W_2$ whose image $\overline{\psi}\in (V^*\otimes W_2)/V$ satisfies
$r_1(\overline{\phi}-\overline{\psi})=0$. Lemma  \ref{le1pourcar0} then  gives that $\overline{\phi}=\overline{\psi}$, which proves Proposition \ref{propcarrenul}(i).
\end{proof}
\begin{proof}[Proof of Proposition \ref{propcarrenul}(ii)] The proof is very similar.
We now work with $V_2={\rm Sym}^2V/q=W_2/q$, and with the  space of infinitesimal invariants
\begin{eqnarray}\label{eqI2V2704bis} I_{2,V}:=(\bigwedge^2V^*\otimes \bigwedge^2V_2)/M_{2,V}(V\cdot (V^*\otimes V_2)).
\end{eqnarray}  We now study the linear map
$$\overline{M}_{2,V}: {\rm Sym}^2((V^*\otimes V_2)/V)\rightarrow I_{2,V}$$
of (\ref{eqM2barpourV}),
obtained by factorization of the $2\times 2$-minor map $$M_{2,V}:{\rm Sym}^2(V^*\otimes V_2)\rightarrow \bigwedge^2V^*\otimes \bigwedge^2V_2.$$
The other adjustments that we have to make are the following: using Lemma  \ref{leverni}, we now get the following analogue
\begin{eqnarray}\label{eqroneprime}r'_1: (V^*\otimes V_2)/V\rightarrow  H^0(G,\mathcal{H}om(\mathcal{S},{\rm Sym}^2\mathcal{E}/q)) \end{eqnarray}
of $r_1$,
where we use the notation $q$ for its natural image in ${\rm Sym}^2\mathcal{E}$ under  restriction. The coherent sheaf ${\rm Sym}^2\mathcal{E}/q$ is not locally free anymore on $G$, but it is locally free on the open set $G^0\subset G$ parameterizing   lines not contained in $\{q=0\}$. We recall that we have now $g=4$ so the rank of $\mathcal{E}$ is $2$. In particular, the codimension of $G\setminus G^0$ in $G$ is $3$. On $G^0$, the $2\times 2$-minors map $m_2$ of (\ref{eqmtwo}) is now replaced with the quadratic map
\begin{eqnarray}\label{eqmtwoprime} m'_2:\mathcal{H}om(\mathcal{S},{\rm Sym}^2\mathcal{E}/q)\rightarrow  \mathcal{H}om(\bigwedge^2\mathcal{S},\bigwedge^2({\rm Sym}^2\mathcal{E}/q)),\\
\nonumber
m'_2(\psi):=\wedge^2\psi.\end{eqnarray}
(We will use the same notation $m'_2$  for the induced map at the level of global sections). We also get
 a linear map
\begin{eqnarray}\label{eqrtwoprime}r'_2: I_{2,V}\rightarrow H^0(G^0,\mathcal{H}om(\bigwedge^2\mathcal{S},\bigwedge^2({\rm Sym}^2\mathcal{E}/q)))\end{eqnarray}
such that, for any $\overline{\phi}\in (V^*\otimes V_2)/V$,
\begin{eqnarray}\label{eqfcatorization1} r'_2(\overline{M}_{2,V}(\overline{\phi}))=m'_2(r'_1(\overline{\phi}))
\,\,{\rm in}\,\, H^0(G^0,\mathcal{H}om(\bigwedge^2\mathcal{S}, \bigwedge^2({\rm Sym}^2\mathcal{E}/q))).\end{eqnarray}

We now prove two  lemmas:
\begin{lemm} \label{le1pourcar0prime} The morphism $r'_1$ of (\ref{eqroneprime}) is injective.
\end{lemm}
\begin{proof} Let $\phi\in V^*\otimes W_2$ be such that its image $\overline{\phi}\in I_{1,V}$ satisfies $r'_1(\overline{\phi})=0$. Denoting by $\psi$ the  image of $\phi$ in $(V^*\otimes W_2)/V$,  this implies that the  morphism
$$ r_1(\psi)\in H^0(G, \mathcal{H}om(\mathcal{S}, {\rm Sym}^2\mathcal{E}))$$
has its image contained in $q\mathcal{O}_G\subset {\rm Sym}^2\mathcal{E}$, hence has everywhere rank $\leq 1$.  By Lemma \ref{le2pourcar0}, there exists a rank $1$ element  $w\in V^*\otimes W_2$ such that
its image in $H^0(G, \mathcal{H}om(\mathcal{S}, {\rm Sym}^2\mathcal{E}))$ equals $r_1(\psi)$. One has
$w=v\otimes q'$ for some element $q'\in S^2V$, and
in fact $q'$ must be proportional to $q$ since it becomes proportional to $q$ in  ${\rm Sym}^2E$, for a general  quotient $E$ of $V$ of dimension $2$. Hence one concludes by Lemma \ref{le1pourcar0} that, modulo $V\subset V^*\otimes W_2$, $\psi$ belongs to $V^*\otimes q\subset V^*\otimes W_2$. So modulo $q$, $\psi$ belongs to $V\subset V^*\otimes V_2$ and the lemma is proved.
\end{proof}
\begin{lemm} \label{le2pourcar0prime} If $\alpha\in H^0(G^0,\mathcal{H}om(\mathcal{S}, {\rm Sym}^2\mathcal{E}/q))$ satisfies
\begin{eqnarray}\label{eqcar0surG1} m'_2(\alpha)=0\,\,{\rm in}\,\, H^0(G^0,\mathcal{H}om(\bigwedge^2{S}, \bigwedge^2({\rm Sym}^2\mathcal{E}/q))),\end{eqnarray}
then there exists a rank $1$ element
$\psi\in V^*\otimes V_2$ whose image in $H^0(G^0,\mathcal{H}om(\mathcal{S},{\rm Sym}^2\mathcal{E}/q))$ equals $\alpha$.
\end{lemm}
\begin{proof} Our assumption is that $\alpha: \mathcal{S}\rightarrow  {\rm Sym}^2\mathcal{E}/q$ has rank $\leq 1$ everywhere, hence, for some line bundle $\mathcal{L}$ on $G$, it factors as $i\circ \beta$, where
$$\beta: \mathcal{S}\rightarrow \mathcal{L}$$
is defined on $G^0$ and $i:\mathcal{L}\rightarrow {\rm Sym}^2\mathcal{E}/q$ is an injective morphism.   As in the proof of Lemma \ref{le2pourcar0}, we discuss the possible values of $\mathcal{L}$. It cannot be strictly negative since
 $$H^0(G^0,\mathcal{S}^*\otimes \mathcal{L})=H^0(G,\mathcal{S}^*\otimes \mathcal{L})$$ is nonzero, because it contains $\beta$. Next, $H^0(G^0,\mathcal{L}^{-1}\otimes ({\rm Sym}^2\mathcal{E}/q))\not=0$ because it contains $i$. It follows that \begin{eqnarray}\label{eqaeroport}H^0(G^0,\mathcal{L}^{-1}\otimes {\rm Sym}^2\mathcal{E})\not=0.\end{eqnarray} Indeed, $H^1(G^0,\mathcal{L}^{-1})=0$, because ${\rm codim}(G\setminus G^0\subset G)\geq 3$, so the claim follows from the exact sequence
 $$0\rightarrow \mathcal{L}^{-1}\rightarrow \mathcal{L}^{-1}\otimes {\rm Sym}^2\mathcal{E}\rightarrow \mathcal{L}^{-1}\otimes ({\rm Sym}^2\mathcal{E}/q)\rightarrow 0$$
 on $G^0$. As
 $$H^0(G^0,\mathcal{L}^{-1}\otimes {\rm Sym}^2\mathcal{E})=H^0(G,\mathcal{L}^{-1}\otimes {\rm Sym}^2\mathcal{E}),$$  we conclude from (\ref{eqaeroport}) that $H^0(G,\mathcal{L}^{-1}\otimes {\rm Sym}^2\mathcal{E})\not=0$ and thus $\mathcal{L}$ cannot be strictly positive. Thus the line bundle $\mathcal{L}$ is trivial, and we conclude exactly as in the proof of Lemma \ref{le2pourcar0}, using now the facts that $H^0(G,\mathcal{S}^*)=V^*$ and $H^0(G^0,Sym^2\mathcal{E}/q)=W_2/q=V_2$.
\end{proof}
The proof of Proposition \ref{propcarrenul}(ii)  follows  from these two lemmas exactly as before.
\end{proof}
\subsection{Proof of Theorem \ref{theoabelian}}
We recall that Theorem \ref{theoabelian} is the following statement.
\begin{theo}\label{theoabeliantexte} (i)  Let $A $ be a very general abelian variety of dimension at least $4$. Then, if $D\in {\rm CH}^1(A)_{\rm hom}={\rm Pic}^0(A)$ satisfies  $D^2=0$ in ${\rm CH}^2(A)$, $D$ is of torsion in ${\rm Pic}^0(A)$.

(ii)  Let $C$ be a very general curve of genus $ 4$. Then if $D\in {\rm Pic}^0(C)={\rm Pic}^0(J(C))={\rm CH}^1(J(C))$ satisfies  $D^2=0$ in ${\rm CH}^2(J(C))$, $D$ is of torsion in ${\rm Pic}^0(C)$.
\end{theo}
The proofs of Theorem \ref{theoabeliantexte}(i) and Theorem \ref{theoabeliantexte}(ii) are very similar, with some technical differences that are essentially already taken care in Proposition \ref{propcarrenul}(i) and (ii), so we  give them in parallel. When a different argument is needed for (i) and (ii), we will refer to case (i) and case (ii).

In case (i), we consider the moduli space $\mathcal{A}_{g}$ of principally polarized abelian varieties, and consider the universal abelian variety $J_g\rightarrow \mathcal{A}_{g}$ with associated (isomorphic) dual fibration $$\widehat{J}_g:={\rm Pic}^0(J/ \mathcal{A}_{g})\rightarrow \mathcal{A}_{g}.$$
In case (ii), we consider  the moduli space $\mathcal{M}_{4}$ of genus $4$ curves, the universal Jacobian
${J}\rightarrow \mathcal{M}_4$ and its dual $\widehat{J}\rightarrow \mathcal{M}_4$ .

We  start recalling  the following general fact (see \cite{voisinchow}).
\begin{lemm} For any abelian fibration $J\rightarrow B$,  the set of points $(b,d), \, d \in \widehat{J}_{b}= {\rm Pic}^0(J_b)={\rm CH}^1(J_{b})_{\rm hom}$, such that
 $d^2=0$ in ${\rm CH}^2(J_{b})$, is a countable union of closed algebraic subsets of $ \widehat{J}$.
\end{lemm}
This set  contains the torsion points in the fibers, by the Bloch-Roitman  theorem on torsion codimension $2$ algebraic cycles (see \cite{bloch}, \cite{roitman}), because the cycles $d^2$ above have trivial Abel-Jacobi invariant. In case (i), it is also known for $g\geq3$ (see \cite{voisinchowab}) that any irreducible algebraic component of this set dominating $\mathcal{A}_{g}$ is generically finite over it. We fix such a dominating component $M$ and denote respectively  by $J_{g,M}\rightarrow M,\,\widehat{J}_{g,M}\rightarrow M$ the fibered products
$J_g\times_{\mathcal{A}_{g}}M,\,\widehat{J}_g\times_{\mathcal{A}_{g}}M$. By construction, there  is a section $M\rightarrow \widehat{J}_{g,M}$, which (possibly after finite base change) gives a divisor $D\in {\rm CH}^1(J_{g,M})$ which is cohomologous to $0$ along the fibers $J_{g,m}$ and satisfies  for any $m\in M$ \begin{eqnarray}
\label{eqdsquarezero} D_m^2=0\,\,{\rm in}\,\, {\rm CH}^2(J_{g,m}).
\end{eqnarray}
The contents of Theorem \ref{theoabeliantexte}(i) is that,  under assumption (\ref{eqdsquarezero}), $D$ is of torsion along the fibers $J_{g,m}$.
In case (ii), we can make exactly the same construction, and  we get this time  a generically finite cover $M\rightarrow \mathcal{M}_4$, and a section $M\rightarrow \widehat{J}_{M}:=\widehat{J}\times_{\mathcal{M}_4}M$, inducing a divisor
$D$ of $J_{M}:=J\times_{\mathcal{M}_4}M$, which is homologous to zero on the fibers $J_{m}$ and satisfies, for any $m\in M$

\begin{eqnarray}
\label{eqdsquarezero(ii)} D_m^2=0\,\,{\rm in}\,\, {\rm CH}^2(J_{m}).
\end{eqnarray}

\begin{proof}[Proof of Theorem \ref{theoabelian}]
With the same notations as above, the cycle $D^2\in {\rm CH}^2(J_{g,M})$, being rationally equivalent to $0$ in the fibers of $J_M\rightarrow M$, has trivial infinitesimal invariant (see Theorem \ref{theomumfordvar}), that is:
\begin{eqnarray}\label{eqvaninint} \delta {D^2}(m)=0\,\,{\rm in}\,H^2(\Omega_{J_{g,M}\mid J_{g,m}}^2)
\end{eqnarray}
for a general $m\in M$.
By compatibility of the cycle class with the intersection and cup-products,  we have
\begin{eqnarray}\label{eqvaninint1} \delta {D^2}(m)=(\delta D(m))^2\,\,{\rm in}\,H^2(\Omega_{J_{g,M}\mid J_{g,m}}^2).
\end{eqnarray}

In case (i), as discussed in the previous section, the infinitesimal invariant $\delta D(m)\in I_{1,W}$ thus satisfies $\overline{M}_{2,W}(\delta D(m))=0$ in $I_{2,W}$.
According to Proposition \ref{propcarrenul}(i), this implies that the infinitesimal invariant
$$\delta D(m)\in (V^*\otimes W_2)/V$$ at a general point $m\in M$ has a unique lift $$\widetilde{\delta {D}(m)}\in V^*\otimes W_2$$
which is of rank $\leq 1$. Note that the case  where the rank is $0$ at the general point is already treated in Lemma  \ref{lepourlatorsion}. When the generic rank is $1$,  we can restrict to the dense Zariski open set of $M$ where this rank is exactly $1$.  In this case, we have:
\begin{claim}\label{claim2904} The normal function $\nu_D\in \Gamma(\mathcal{J}_g^\vee)$ admits  local  lifts
$\tilde{\nu}\in \Gamma(\mathcal{H}^1)$ having the property that, at any  point $m\in M$
\begin{eqnarray}\label{eqdnurang1} (\nabla\tilde{\nu})^{0,1}_m\in H^{0,1}(J_{g,m})\otimes \Omega_{M,m}=V^*\otimes W_2
\end{eqnarray}
is a rank $1$ tensor in $V^*\otimes W_2$.
\end{claim}
 In (\ref{eqdnurang1}), $(\nabla\tilde{\nu})^{0,1}_m$ denotes the $(0,1)$-part of $\nabla\tilde{\nu}_m\in H^1(J_{g,m},\mathbb{C})\otimes\Omega_{M,m}$.
\begin{proof}[Proof of Claim \ref{claim2904}]  We  choose a first local lift $\tilde{\nu}_0$ of $\nu_D$.
By definition, $\delta\nu_{D,m}=\delta D(m)$ is obtained by considering $(\nabla\tilde{\nu}_0)^{0,1}_m\in H^{0,1}(J_{g,m})\otimes \Omega_{M,m}$ modulo $\overline{\nabla}H^{1,0}(J_{g,m})$. Thus we have
\begin{eqnarray}\label{eqfut2804} (\nabla\tilde{\nu}_0)^{0,1}=
\widetilde{\delta D}+\overline{\nabla}\eta\,\,{\rm in}\,\,\Gamma(\mathcal{H}^{0,1}\otimes \Omega_M),\end{eqnarray}
where $\eta$ is a section of $\mathcal{H}^{1,0}$ and  $\widetilde{\delta D}$ is the section of  $\mathcal{H}^{0,1}\otimes \Omega_M$ defined by $\widetilde{\delta D}(m):= \widetilde{\delta D(m)}$.
By (\ref{eqfut2804}), the lift  $\tilde{\nu}:=\tilde{\nu}_0-\eta$ then satisfies
\begin{eqnarray}\label{eqfut2804bis} (\nabla\tilde{\nu})^{0,1}= \widetilde{\delta D}\,\,{\rm in}\,\,\Gamma(\mathcal{H}^{0,1}\otimes \Omega_M),\end{eqnarray}
where  $\widetilde{\delta D}$ is a rank $1$  section of $\mathcal{H}^{0,1}\otimes \Omega_M$.
\end{proof}

In case (ii), the infinitesimal invariant $\delta D\in I_{1,V}$  satisfies $\overline{M}_{2,V}(\delta D)=0$ in $I_{2,V}$ and, using Proposition \ref{propcarrenul}(ii) and arguing exactly as above, we get, after shrinking $M$,  that the normal function $\nu_D\in \Gamma(\mathcal{J}^\vee)$ admits  local  lifts
$\tilde{\nu}\in \Gamma(\mathcal{H}^1)$ having the property that, at any  point $m\in M$
\begin{eqnarray}\label{eqdnurang1prime} (\nabla\tilde{\nu})^{0,1}\in H^{0,1}(J_{m})\otimes \Omega_{M,m}=V^*\otimes V_2
\end{eqnarray}
is a rank $1$ tensor in $V^*\otimes V_2$.

We will  prove below  the following result.
\begin{prop}\label{propourcrcutheoab}  Let  $M$ be a smooth quasi-projective variety over $\mathbb{C}$ and let $J\rightarrow M$ be a family of abelian varieties, which is either the universal abelian variety of dimension $g\geq 4$ over $M$, for some generically finite morphism $M\rightarrow \mathcal{A}_g$  (case (i)), or the universal Jacobian over $M$, for some generically finite morphism $M\rightarrow \mathcal{M}_4$ (case (ii)). Let  $D\in {\rm Pic}(J)$ be homologous to $0$ on fibers. Assume that the associated normal function $\nu_D$ has the property that it admits local lifts $\tilde{\nu}\in \Gamma(\mathcal{H}^1)$ such that the $(0,1)$-part
$$(\nabla\tilde{\nu})^{0,1}\in \Gamma(\mathcal{H}^{0,1}\otimes \Omega_M)$$ of $\nabla\tilde{\nu}$ has rank $\leq 1$ at any point of $M$. Then $\nu_D$ admits local lifts $\tilde{\nu}\in \Gamma(\mathcal{H}^1)$ which are  horizontal, that is, $\nabla\tilde\nu=0$.
\end{prop}
  Proposition \ref{propourcrcutheoab} and
 Lemma \ref{lepourlatorsion}  then   imply  that $\nu_D$ is of torsion, which concludes the proof of Theorem \ref{theoabeliantexte}.
\end{proof}

\begin{proof}[Proof of Proposition \ref{propourcrcutheoab}]
We argue by contradiction and assume furthermore that  over a dense  Zariski open set of $M$, $\nu_D$  admits local lifts $\tilde{\nu}\in \Gamma(\mathcal{H}^1)$ such that the $(0,1)$-part
$$(\nabla\tilde{\nu})^{0,1}\in \Gamma(\mathcal{H}^{0,1}\otimes \Omega_M)$$ of $\nabla\tilde{\nu}$ has rank exactly $1$.
\begin{lemm} \label{lepourproppourcruc}   The following hold.

(i) The normal function $\nu_D$ admits local  lifts
$\tilde{\nu}\in \Gamma(\mathcal{H}^1)$ with the property that
\begin{eqnarray}\label{eqnablatildenurank1} \nabla{\tilde{\nu}}=\beta\otimes \alpha\,\,{\rm in}\,\,\Gamma(\mathcal{H}^1\otimes \Omega_M),
\end{eqnarray}
where $\beta\in \Gamma(\mathcal{H}^1)$, $\alpha\in\Gamma(\Omega_M)$.

(ii)
Furthermore,  the $1$-form $\alpha$ defines an integrable distribution.
\end{lemm}
\begin{rema}{\rm   Our  assumption is that, near a general point of $M$, we can choose  local lifts $\tilde{\nu}_0$ satisfying the property that  the $(0,1)$-part $(\nabla\tilde{\nu}_0)^{0,1}\in \Gamma(\mathcal{H}^{0,1}\otimes\Omega_M)$ has rank $1$, hence is locally decomposable as in (\ref{eqnablatildenurank1}). Statement (i) is that we can achieve this local  decomposability for the full differential  $\nabla\tilde{\nu}$ for an adequate choice of $\tilde{\nu}$.}
\end{rema}
\begin{proof}[Proof of Lemma \ref{lepourproppourcruc}]  By assumption, near a general point of $M$, we can choose  local lifts $\tilde{\nu}_0$  such that
\begin{eqnarray}\label{eqnablatildenurank1approx}\nabla\tilde{\nu}_0=\beta_0\otimes \alpha+\beta',
\end{eqnarray}
Where $\beta_0\in\Gamma(\mathcal{H}^1)$, $\alpha$ is a holomorphic $1$-form,  and $\beta'\in \Gamma(\mathcal{H}^{1,0}\otimes \Omega_M)$ is a local section.
Differentiating (\ref{eqnablatildenurank1approx}) and using $\nabla(\nabla\tilde{\nu}_0)=0$, we get the local equality
\begin{eqnarray}\label{eqnablatildenurank1approx2}\nabla\beta'=-\nabla\beta_0\wedge\alpha-\beta_0\otimes d\alpha\,\,{\rm in}\,\, \Gamma(\mathcal{H}^1\otimes\Omega_M^2).
\end{eqnarray}
Projecting (\ref{eqnablatildenurank1approx2}) in $\Gamma(\mathcal{H}^{0,1}\otimes\Omega_M^2)$ gives

\begin{eqnarray}\label{eqnablatildenurank1approx3}\overline{\nabla}_1(\beta')=
-(\nabla\beta_0)^{0,1}\wedge\alpha-\beta_0^{0,1}\otimes d\alpha\,\,{\rm in}\,\, \Gamma(\mathcal{H}^{0,1}\otimes\Omega_M^2),
\end{eqnarray}
where we recall that
\begin{eqnarray}\label{eqnotde165} \overline{\nabla}_1:\mathcal{H}^{1,0}\otimes\Omega_M\rightarrow \mathcal{H}^{0,1}\otimes\Omega_M^2 \end{eqnarray} is induced by $\overline{\nabla}:\mathcal{H}^{1,0}\rightarrow \mathcal{H}^{0,1}\otimes\Omega_M$ and is defined
by $\overline{\nabla}_1(\eta\otimes \gamma)=\overline{\nabla}\eta\wedge\gamma$.
We now make the following
\begin{claim} \label{claim1404} Near  a general point of $M$, Equation (\ref{eqnablatildenurank1approx3}) implies that
$\beta'=\beta''\otimes \alpha$ for some local section $\beta''$ of $\mathcal{H}^{1,0}$.
\end{claim}
Assuming the claim, (\ref{eqnablatildenurank1approx}) now becomes
\begin{eqnarray}\label{eqnablatildenurank1approx5}\nabla\tilde{\nu}_0=(\beta_0+\beta'')\otimes \alpha\,\,{\rm in}\,\, \Gamma(\mathcal{H}^1\otimes\Omega_M),
\end{eqnarray}
which proves Lemma \ref{lepourproppourcruc}(i), with $\beta=\beta_0+\beta''$. Finally we observe that by differentiating (\ref{eqnablatildenurank1}), we get
\begin{eqnarray}\label{eqnablatildenurank1derive} 0=\beta\otimes d\alpha+ \nabla\beta\wedge \alpha\,\,{\rm in}\,\,\Gamma(\mathcal{H}^1\otimes \Omega_M^2),
\end{eqnarray}
As $\beta\not=0$, this implies that $\alpha\wedge d\alpha=0$, so that the distribution defined by $\alpha$ is integrable. This proves Lemma \ref{lepourproppourcruc}(ii) assuming Claim \ref{claim1404}.
\end{proof}

\begin{proof}[Proof of Claim \ref{claim1404}]  We recall that  $M$ is a generically finite cover of either $\mathcal{A}_g$, with $g\geq 4$, (Case (i)) or $\mathcal{M}_4$ (case (ii)). We restrict to the dense Zariski open set of $M$ where $M$ is \'{e}tale over  $\mathcal{A}_g$ or $\mathcal{M}_4$. In case (i), keeping  the same notation as before, so that in particular $V=H^{1,0}(A)$ for $[A]\in\mathcal{A}_g$,  the map $\overline{\nabla}$ at the point $[A]$ is the multiplication map
$$\mu_W:  V\rightarrow  {\rm Hom}(V,W_2),$$
with $W_2:={\rm Sym}^2V$.
In case (ii), with the notation $V=H^{1,0}(C)$ for $[C]\in  \mathcal{M}_4$ a general curve, the
 map $\overline{\nabla}$ at the point $[C]$ is the multiplication map
$$\mu_V:  V\rightarrow  {\rm Hom}(V,V_2),$$
with $V_2=H^0(C,2K_C)=W_2/q$, the quadratic form $q$ being  nondegenerate.

In case (i), let $\mu_{W,1}: V\otimes W_2\rightarrow V^*\otimes\bigwedge^2W_2$ be defined by
$$\mu_{W,1}(\omega):=\mu_W(\beta) \wedge \gamma\in V^*\otimes \bigwedge^2W_2$$ for $\omega=\beta\otimes \gamma\in V\otimes W_2$ (this is the map $\overline{\nabla}_1$ of (\ref{eqnotde165})). We have to  show that for
any element
$ \omega\in V\otimes W_2$ such that
\begin{eqnarray}\label{eqpourmuWdu1404} \mu_{W,1}(\omega)=v_0^*\otimes \alpha_2+ \omega'\wedge  \alpha,\end{eqnarray}
for some  $v_0^*\in V^*,\,\alpha_2\in \bigwedge^2W_2$, $\omega'\in V^*\otimes W_2$, and $\alpha\in W_2$, one has
$\omega=v\otimes \alpha$ for some $v\in V$.
Let $$\overline{W}_2:=W_2/\langle \alpha\rangle.$$ Let $V_0\subset V$ be the hyperplane defined by $v_0^*$. Equation (\ref{eqpourmuWdu1404}) is equivalent to the vanishing of
$\mu_{W,1}(\omega)\in V^*\otimes \bigwedge^2W_2$ in the quotient ${\rm Hom}(V_0,\bigwedge^2\overline{W}_2)$. Denoting by $\overline{\omega}\in V\otimes \overline{W}_2$  the projection of $\omega\in V\otimes W_2$, we have  to prove that $\overline{\omega}=0$.

We write $\overline{\omega}=\sum_{i=1}^rv_i\otimes w_i$, with $w_i$ independent in $\overline{W}_2$, and $v_i$ independent in $V$, for some  $r\leq {\rm dim}\, V$.
Equation (\ref{eqpourmuWdu1404}) tells that, for any $u\in V_0$,
\begin{eqnarray}\label{eqdu15avril}  \sum_{i=1}^r uv_i\wedge  w_i=0\,\,{\rm in}\,\,\bigwedge^2 \overline{W}_2.
\end{eqnarray}

As the $w_i$  are independent in $\overline{W}_2$,  (\ref{eqdu15avril}) implies that $$\overline{uv_i}\in \langle {w}_i,\,i\geq 1\rangle\subset \overline{W}_2,$$ where $\overline{uv_i}$ denotes the projection of $uv_i$ in $\overline{W}_2$. Equivalently, choosing lifts $\tilde{w}_i\in W_2$ of $w_i$, we have
\begin{eqnarray}\label{eqdu15avrilbis}uv_i\in  \langle  \alpha,\tilde{w}_i,\,i\geq 1\rangle\subset {W}_2,\,\,\forall u\in V_0.
\end{eqnarray}
Condition (\ref{eqdu15avrilbis}) says in particular that the image of $V_0\otimes \langle v_i\rangle$ in ${W}_2={\rm Sym}^2V$ is contained in the space $\langle  \alpha,\tilde{w}_i,\,i\geq 1\rangle$ which has dimension $r+1$. As ${\rm dim}\,V_0=g-1$ and $g\geq 4$, we easily get a contradiction unless $r=0$ (see below for the precise argument).

In case (ii), we have similarly  to   show that for
any element
$ \omega\in V\otimes V_2$ such that
\begin{eqnarray}\label{eqpourmuWdu1404bis} \mu_{V,1}(\omega)=v_0^*\otimes \alpha_2+ \omega'\wedge  \alpha,\end{eqnarray}
for some  $v_0^*\in V^*,\,\alpha_2\in \bigwedge^2V_2$, $\omega'\in V^*\otimes V_2$, and $\alpha\in V_2$, one has
$\omega=v\otimes \alpha$ for some $v\in V$. We argue  as before, write $\overline{V}_2:=V_2/\alpha$, and $$\overline{\omega}=\sum_{i=1}^r v_i\otimes {w}_i\,\,{\rm in}\,\,V\otimes \overline{V}_2,$$
with $r$ minimal. The vanishing of $\mu_{V,1}(\omega)$ in $V_0^*\otimes \bigwedge^2(V_2/\alpha)$ then gives as above  that, choosing lifts $\tilde{w}_i\in V_2$ of $w_i\in V_2/\alpha$,
$$  uv_i\in  \langle  \alpha,\tilde{w}_i,\,i\geq 1\rangle\subset {V}_2,\,\,\forall u\in V_0.$$
We then  see that it suffices  to prove the following statement:

\vspace{0.5cm}

(*) {\it Let  $V$ be of vector space of  dimension $4$ and $V_2={\rm Sym}^2V/q$, for some nondegenerate quadratic form $q$. Let $V_0\subset V$ be a hyperplane, and let
$$ H_r\subset  V,\,H_{r+1}\subset V_2$$ be  vector subspaces of respective  dimensions $r$,  $r+1$. Then, if
$$V_0\cdot H_r\subset  H_{r+1},$$ one has $r=0$.}

Indeed, we apply statement (*) to $H_r=\langle v_i\rangle\subset V,\,H_{r+1}=\langle \tilde{w}_i,\,\alpha\rangle \subset V_2$, where the $\tilde{w}_i$ are lifts of ${w}_i$ in $V_2$. Then we conclude that $r=0$.

\vspace{0.5cm}

We finally  prove (*): as the multiplication map $V\otimes V\rightarrow V_2$ does not vanish on nonzero decomposable tensors, the Hopf Lemma says that ${\rm dim}\,(V_0\cdot H_r)\geq g+r-2$ if $r>0$, hence, if $g\geq 4$, we get
$${\rm dim}\,(V_0\cdot H_r)\geq r+2.$$  Thus  we must have $r=0$.
\end{proof}

Coming back to the proof of Proposition \ref{propourcrcutheoab}, we now observe the following facts.
\begin{lemm} \label{leuniqueglobal}  Notation and assumptions being as in Lemma \ref{lepourproppourcruc}, we have

(i) The local lift $\tilde{\nu}$ of Lemma \ref{lepourproppourcruc} is unique modulo horizontal  sections of $\mathcal{H}^1$. It follows that $\nabla\tilde{\nu}\underset{\rm loc}{=}\beta\otimes \alpha$ is  a global section of $\mathcal{H}^1\otimes \Omega_M$. In particular, the holomorphic  distribution induced by $\alpha$ is globally defined, at least on the Zariski open set of $M$ where $\delta\nu\not=0$.

(ii) Up to passing to a dense Zariski open set of $M$ if necessary, one has locally defined holomorphic functions $f$ on $M$ such that
\begin{eqnarray}\label{eqnablabetaalpaha} d(f\alpha)=0,\,\, \nabla (f^{-1}\beta)=\beta'\otimes \alpha,\,\,
\end{eqnarray} for some local section $\beta'$ of $\mathcal{H}^1$.
\end{lemm}
\begin{proof} Any other local  lift $\tilde{\nu}'$ is of  the form
$\tilde{\nu}'=\tilde{\nu}+\gamma_\mathbb{Z}+\gamma^{1,0}$, where $\gamma_\mathbb{Z}$ is a section of $H^1_\mathbb{Z}$ and
$\gamma^{1,0}$ is a local section of $\mathcal{H}^{1,0}$. We then have
$$(\nabla\tilde{\nu}')^{0,1}=(\nabla\tilde{\nu})^{0,1}+\overline{\nabla}\gamma^{1,0}.$$
If  $\tilde{\nu}'$, $\tilde{\nu}$ both satisfy the property that $(\nabla\tilde{\nu}')^{0,1}$, $(\nabla\tilde{\nu})^{0,1}$ are rank $1$ elements of $\mathcal{H}^{0,1}\otimes \Omega_M$, the uniqueness in Proposition \ref{propcarrenul} implies that $\overline{\nabla}\gamma^{1,0}=0$, hence that $\gamma^{1,0}=0$ by injectivity of $\overline{\nabla}$. This proves (i).

  By Lemma \ref{lepourproppourcruc}(ii), the distribution determined by $\alpha$  is integrable, which means that, on the Zariski open set where it is well-defined and nonzero, $\alpha$ can be  locally (in the Euclidean topology)  multiplied by an adequate holomorphic function $f$  in such a way that $d(f\alpha)=0$. Equation (\ref{eqnablatildenurank1derive}) then gives
\begin{eqnarray}\label{eqnablatildenurank1newplus} 0= \nabla (f^{-1}\beta)\wedge f\alpha\,\,{\rm in}\,\,\Gamma(\mathcal{H}^1\otimes \Omega_M^2),
\end{eqnarray}
which  tells that $\nabla(f^{-1}\beta)=\beta'\otimes\alpha$ for some local section $\beta'\in \Gamma(\mathcal{H}^1)$,  proving (ii).
\end{proof}
We next prove  the following result:
\begin{lemm} \label{lealgbetaomega} The global holomorphic section $\nabla\tilde{\nu}\underset{loc}{=}\beta\otimes \alpha$ of $\mathcal{H}^1\otimes \Omega_M$ constructed in Lemma \ref{leuniqueglobal}(i) is algebraic.
\end{lemm}
\begin{proof} We note first that we could have used, instead of  the exact sequence (\ref{eqexcatpourJ}) defining $\mathcal{J}^\vee$, the following exact sequence
\begin{eqnarray}\label{eqexactaveczenbas} 0\rightarrow \mathcal{H}^{1,0}\rightarrow  \mathcal{H}^{1}/H^1_\mathbb{Z}\rightarrow \mathcal{J}^\vee\rightarrow 0.
\end{eqnarray}
Then the local lifts $\tilde{\nu}$ of $\nu$ are local sections of  $\mathcal{H}^{1}/H^1_\mathbb{Z}$  defined  modulo local sections of $\mathcal{H}^{1,0}$, and  we can use the Gauss-Manin connection
\begin{eqnarray}\label{eqnalabmodz}\nabla: \mathcal{H}^1/H^1_\mathbb{Z}\rightarrow \mathcal{H}^1\otimes\Omega_M.\end{eqnarray}
Next we observe that the exact sequence (\ref{eqexactaveczenbas}) has an algebraic version, in which $\mathcal{J}^\vee$ is replaced with the sheaf of algebraic sections of ${\rm Pic}^0(J_M/M)\rightarrow M$, $\mathcal{H}^{1,0}=R^0\pi_*\Omega_{J/M}$ is the algebraic version of the Hodge bundle, constructed using algebraic relative differentials, and the algebraic version of the term  $\mathcal{H}^{1}/H^1_\mathbb{Z}$ in the middle  is the sheaf $\mathcal{P}$ of algebraic sections of the abelian group scheme  over $M$ parameterizing line bundles on fibers of $\pi:J_M\rightarrow M$ equipped   with an (integrable) algebraic  connection.  Via the Riemann-Hilbert correspondence, the  complex points  of this group scheme over $t\in M$ are representations of $\pi_1(J_t)$ with values in $\mathbb{C}^*$, that is, elements of  $$H^1(J_t,\mathbb{C}^*)=H^1(J_t,\mathbb{C})/H^1(J_t,\mathbb{Z}).$$ (Note however  that  the later viewpoint is transcendental.) The connection (\ref{eqnalabmodz}) is algebraically defined  on the algebraic group scheme $\mathcal{P}$: if we have a line bundle with algebraic (integrable) connection $(L,\nabla_L)$ on a smooth projective   variety $X$ and a first order deformation
$X_\epsilon\rightarrow \Delta_\epsilon$ of $X$, where $\Delta_\epsilon:={\rm Spec}\,(\mathbb{C}[\epsilon]/\epsilon^2)$, then there exists a unique pair
$(L_\epsilon, \nabla_{L_\epsilon})$ consisting of a line bundle with connection on $X_\epsilon$, restricting to $(L,\nabla_L)$ on the central fiber.  If $P_1(L)$ is the bundle of $1$-jets of $L$, fitting in an extension
$$0\rightarrow \Omega_X(L)\rightarrow P_1(L) \rightarrow L\rightarrow0,$$
and $T_L:=P_1(L)^*\otimes L$,
there is an exact sequence
\begin{eqnarray}\label{eqnatexactdu7ju}0\rightarrow \mathcal{O}_X\rightarrow T_L \rightarrow T_X\rightarrow0,\end{eqnarray}
and the connection $\nabla$ gives a splitting of this exact sequence, and thus a canonical splitting
\begin{eqnarray}\label{eqsplitting}  H^1(X,T_L)\cong H^1(X,\mathcal{O}_X)\bigoplus H^1(X,T_X).\end{eqnarray}
The vector space $H^1(X,T_L)$ parameterizes the first order deformations of the pair $(X,L)$ and the natural map $H^1(X,T_L)\rightarrow H^1(X,T_X)$ induced by (\ref{eqnatexactdu7ju}) is the forgetful map. The splitting (\ref{eqsplitting}) induces a section $H^1(X,T_X)\rightarrow H^1(X,T_L)$ which provides the desired first order extension.
 These constructions are classical (see \cite{mazurmessing}, \cite{nitsure}) and can be also seen as a particular case of Simpson's work \cite{simpson} on de Rham moduli spaces.

Having this algebraic version of the exact sequence (\ref{eqexcatpourJ}), we argue  as before, except that our local lifts
$\tilde{\nu}$ of the algebraic section $\nu_D$ of $\mathcal{J}^\vee$ are now algebraic  sections of $\mathcal{P}$ which   are local  for the Zariski topology. The existence and uniqueness of these local lifts $\tilde{\nu}$ satisfying  property (\ref{eqnablatildenurank1}) are  proved as before. It follows that $\nabla\tilde{\nu}$ is algebraic.
\end{proof}

Combining Lemmas \ref{lepourproppourcruc}, \ref{leuniqueglobal} and \ref{lealgbetaomega} together with  Proposition \ref{lecructresgenpourpropourcrcutheoab} stated and proved below, we  conclude  that the local lifts $\tilde{\nu}$ of Lemma \ref{lepourproppourcruc} are in fact horizontal, hence Proposition \ref{propourcrcutheoab} is now fully proved.
\end{proof}

\begin{prop} \label{lecructresgenpourpropourcrcutheoab} Let  $M$ be smooth quasi-projective and $J_M\rightarrow M$ be a family of abelian varieties, which is either the universal abelian variety of dimension $g\geq 2$ over $M$ for some generically finite morphism $M\rightarrow \mathcal{A}_g$, or the universal Jacobian over $M$ for some generically finite morphism $M\rightarrow \mathcal{M}_g$, with $g\geq2$. Let $\omega\in \Gamma(\mathcal{H}^1\otimes \Omega_M)$ be  an algebraic section which is locally decomposable, i.e. locally of the form $\beta\otimes \alpha$, and annihilated by $\nabla$.
 Then $\omega=0$.
\end{prop}
\begin{proof} We assume that $\omega\not=0$ and get a contradiction. Let $\mathcal{H}^1$ be the Hodge bundle with Gauss-Manin  connection $\nabla$. This is an algebraic vector bundle on $M$. Let $p:\mathfrak{F}\rightarrow M$ be the symplectic frame bundle of $\mathcal{H}^1$. This is a principal ${\rm Sp}(2g)$-bundle which  is  equipped with a flat connection $\nabla$. The leaves for this connection are complex submanifolds of  $\mathfrak{F}$ corresponding to the locally constant frames. The locally decomposable algebraic section $\omega\underset{loc}{=}\beta\otimes \alpha$ provides a line subbundle $\mathcal{L}\subset   \mathcal{H}^1$, hence an algebraic subvariety
$$\mathfrak{F}_{\mathcal{L}}\subset  \mathfrak{F}$$
consisting of frames $e_1,\ldots,\,e_{2g}$ with $e_1$ belonging to $\mathcal{L}$ (that is, proportional to $\beta$).
The intersection of $\mathfrak{F}_{\mathcal{L}}$ with the  leaf $\mathfrak{F}_e$  passing through a point $e=(e_1,\ldots,\,e_{2g})\in \mathfrak{F}_{m_0}$,  is described over $M$ by the condition that $e_1(m)$ is proportional to $\beta(m)$ in the fiber $\mathcal{H}^1_{m}:=H^1(J_{m},\mathbb{C})$. By Lemma \ref{leuniqueglobal}(ii), there exist an integrable codimension $1$ algebraic  distribution $\alpha$ on $M$ and   local holomorphic functions $f$ on $M$ such that $f\beta$ is horizontal along the leaves of the distribution $\alpha$. Hence the intersection  $\mathfrak{F}_e\cap  \mathfrak{F}_{\mathcal{L}}$ contains he inverse image   under  $p$ of  the leaf of $\alpha$ passing through $m_0\in M$. In particular, it has codimension $\leq  1$ in $\mathfrak{F}_e$,  while $\mathfrak{F}_{\mathcal{L}}$ has codimension $2g-1$ in $\mathfrak{F}$. It follows that the intersection $\mathfrak{F}_e\cap  \mathfrak{F}_{\mathcal{L}}$ is not transverse if $g\geq 2$. As in both considered cases, the Galois group of the connection is Zariski dense in the symplectic group, we can apply
 \cite[Theorem A]{nagloo}. This tells us  that the images
$p(\mathfrak{F}_e\cap \mathfrak{F}_{\mathcal{L}})$ in $M$ are contained in (proper) $\nabla$-special algebraic varieties (on which the Galois group of the restricted connection is not anymore Zariski dense). As the images $p(\mathfrak{F}_e\cap \mathfrak{F}_{\mathcal{L}})$ contain open sets of the leaves, which are of codimension $1$ in $M$, we conclude that the leaves of $\alpha$ are algebraic, that is,   the distribution $\alpha$ is algebraically integrable. It follows that the distribution $\alpha$  is, at the general point of  $M$, the relative tangent bundle of a rational map  $M\dashrightarrow D$, where $D$ is a curve, and in particular it is defined by a closed rational $1$-form on $M$.
We then  know by Lemma \ref{leuniqueglobal}(ii) that  the class $\beta$ conveniently normalized satisfies $\nabla\beta=0$  along any leaf $M_{\alpha,t}$ of $\alpha$. By Deligne's global invariant cycles theorem, the set of classes globally invariant along the leaf $M_{\alpha,t}$  (which we know to be algebraic) generates at any point $m\in M_{\alpha,t}$  a constant Hodge substructure of $H^1(J_{M,m},\mathbb{Q})$. As, by varying $t$,  $m$ can be taken very general in $M$,  any Hodge substructure of $H^1(J_{M,m},\mathbb{Q})$ is $0$ or equals $H^1(J_{M,m},\mathbb{Q})$. The first case would provide a contradiction with $\omega\not=0$. The second case would say that the moduli map for the family $J_M\rightarrow M$ is constant along the leaves, hence factors through a curve, which is not the case in our two geometric situations.
 \end{proof}

\begin{rema}{\rm Bruno Klingler suggested the following alternative argument replacing the use of the invariant cycles theorem at the end of  the proof given above. Once one knows that the distribution $\alpha$ is algebraically integrable, one gets that the (Zariski closure of the) image of the monodromy representation restricted to the general leaf of $\alpha$ is a proper and normal subgroup of the full algebraic monodromy group, which is ${\rm Sp}(2g)$. As this image  is  a proper subgroup because the leaves are $\nabla$-special, it follows that the image is trivial. We are thus in the second case above.}
\end{rema}

\subsection{Consequences on Pontryagin product of $0$-cycles and orbits for rational equivalence \label{secnewconspont}}
Let $A$ be an abelian variety with origin $0_A$, and $\theta\in{\rm Pic}(A)$ be a polarization, that is,  an ample divisor class. The polarization gives an isogeny of abelian varieties
$$A\rightarrow \widehat{A}:={\rm Pic}^0(A),$$
$$ x\mapsto D_x:=\theta_x -\theta,$$
where $\theta_x$ is the translate of $\theta$ by $x$.  Recall that the abelian group structure
$\mu:A\times A\rightarrow A$ on $A$  allows to define the Pontryagin product $*$ on cycles by the formula
$$z*z':=\mu_*(z\times z').$$
 Denote by $0_A$  the origin of an abelian variety $A$ (it is neeeded if we allude to the group structure of $A$, which appears in the definition of the Pontryagin product). The following statement is proved in \cite{voisinchowab} :
\begin{lemm} \cite[Lemma 1.7]{voisinchowab} One has $D_x^k
= 0 $ in ${\rm CH}^k(A)$ if and only if $(\{x\} - \{0_A\})^{*k} = 0$ in $ {\rm CH}_0(A)$.
\end{lemm}

This lemma allows us to   rephrase Theorem \ref{theoabelian} as follows.
 \begin{coro} \label{remaintro}  For a  very general abelian variety $A$ of dimension $g\geq 4$, the set of points $x\in A$ such that $(\{x\}-\{0_A\})^{*2}=0$ in ${\rm CH}_0(A)$ consists of the torsion points of $A$. The same statement holds  if $A$ is a very general Jacobian in genus $4$.
\end{coro}
Finally, Corollary \ref{remaintro} in turn has the following consequence on orbits for rational equivalence  of elements in the symmetric product $A^{(2)}$ of an abelian variety, that is, the set $O_{z_0}$
of $z\in A^{(2)}$ such that $z$ is rationally equivalent to $z_0$ in $A$.
\begin{coro}\label{coroorbit} If $A$ is  a very general abelian  variety   of dimension $g\geq 4$ or   the Jacobian  of a very general  curve of genus $4$, the orbit
$O_{2\{0_A\}}\subset A^{(2)}$  of the $0$-cycle $2\{0_A\}\in A^{(2)}$ for rational equivalence on $A$ consists only of (effective, degree $2$) $0$-cycles supported on torsion points of $A$.
\end{coro}
 \begin{proof}  By \cite[Proposition 1.9]{voisinchowab}, we know that a rational equivalence relation
 \begin{eqnarray}\label{eqnnewdu14juin} \{x_1\}+\ldots+\{x_k\}= k\{0_A\} \,\,{\rm in}\,\,{\rm CH}_0(A)\end{eqnarray}
 implies that $(\{x_i\}-\{0_A\})^{*k}=0 $  in ${\rm CH}_0(A)$ for $i=1,\ldots,\,k$. Thus, for $k=2$, when $A$ is very general, (\ref{eqnnewdu14juin})    implies that each  $x_i$ is a torsion point of $A$ by Corollary  \ref{remaintro}. \end{proof}

\begin{rema}\label{remag2du910}{\rm The situation is unclear in genus $3$, but for abelian surfaces and more generally for Jacobians $A$ of hyperelliptic curves, there are uncountably many points $x\in A$ such that $(\{x\}-\{0_A\})^{*2}=0$ in ${\rm CH}_0(A)$ (in particular there are nontorsion points). Indeed, let $A=J(C)$ for an hyperelliptic  curve $C\subset A$ and assume that $C$ passes through $0_A$, and that the corresponding point $x_0\in C$ is a Weierstrass point. Then, by the same argument as above,  any $x\in C\subset A$ has the property that  $(\{x\}-\{0_A\})^{*2}=0$ in ${\rm CH}_0(A)$, since, for some $y\in C$, $\{x\}+\{y\}$ is rationally equivalent to $2x_0$ in $C$, hence to $2\{0_A\}$  in $A$.}
\end{rema}

\section{Proof of Theorem \ref{theoquestionpirola}  \label{secpirola}}
This section is devoted to the proof of Theorem \ref{theoquestionpirola}. We start with the following statement, proved in \cite{verni}. For   a smooth curve $C$ of genus $4$, let $\gamma_{P,C}\in {\rm CH}^1(C)$ be the Griffiths-Pirola divisor (which is defined up to sign). We can see also $\gamma_{P,C}$ as an element of $ {\rm CH}^1(C^{(2)})_{\rm hom}$ using the isomorphism ${\rm Pic}^0(C^{(2)})\cong {\rm Pic}^0(C)$.
\begin{theo} \label{proverni}  (\cite{greengriffiths}, \cite{yinq}, \cite[Theorem 1.4]{verni}.) Let $\pi:\mathcal{C}\rightarrow B$ be the universal family of smooth  $(2,3)$-complete intersections in $\mathbb{P}^3$, and let $\pi_2: \mathcal{C}^{(2)/B}\rightarrow B$ be the corresponding family of symmetric products. Let $C_\eta,\,C^{(2)}_\eta$ be the respective  generic fibers of $\pi$, $\pi_2$. Then ${\rm CH}_0(C^{(2)}_\eta)_{\rm hom}\otimes \mathbb{Q}$ is $1$-dimensional, generated over $\mathbb{Q}$ by $\gamma_{P,\eta}^2$.
\end{theo}
More precisely,  earlier arguments   appearing in \cite{voisinhodgebloc} (see also \cite{fulavi}) show that
${\rm CH}_0(C^{(2)}_\eta)_{\rm hom}\otimes \mathbb{Q}$ is at most $1$-dimensional over $\mathbb{Q}$, generated by the Green-Griffiths $0$-cycle of \cite{greengriffiths}, called Faber-Pandharipande $0$-cycle in  \cite{yinq}, which is proved to be nonzero   in \cite{greengriffiths}, \cite{yinq}.  Verni proves in \cite{verni} that $\gamma_{P,\eta}^2\in {\rm CH}_0(C^{(2)}_\eta)_{\rm hom}\otimes \mathbb{Q}$ is different from $0$ (see Theorem \ref{theoverni}), which gives a third proof of the Green-Griffiths nonvanishing result.
Note that, due to the sign ambiguity, $\gamma_{P,\eta}$ is not defined as an element of ${\rm CH}^1({C}_\eta^{(2)})_{\rm hom}$, but only as an element of ${\rm CH}^1({C}^{(2)}_{\eta'})_{\rm hom}$ for a degree $2$ field extension of $\mathbb{C}(\mathcal{M}_4)$. However, its square $\gamma_{P,\eta'}^2$ descends to     an element $\gamma_{P,\eta}^2$ of ${\rm CH}_0(C^{(2)}_\eta)$, although maybe only with rational coefficients (see below for the precise argument).

Let $\pi: \mathcal{C}\rightarrow \mathcal{M}_4$ be the universal curve of genus $4$ and let $J\rightarrow \mathcal{M}_4$ be the corresponding Jacobian fibration. Assume there is a rational  section $\gamma: \mathcal{M}_4\dashrightarrow {K}$ of the associated Kummer fibration
$${K}:=J/\pm{\rm Id}.$$ As the quotient map $J\rightarrow K$  is a double cover, one gets by base changing  under $\gamma$ a  double cover
$$\mathcal{M}_4':=\mathcal{M}_4\times_{K}{J}.$$
The  double cover  $\mathcal{M}_4'$  also  maps to ${J}$, via a map that we denote by $\gamma'$.  Let $$\mathcal{C}':=\mathcal{C}\times_{\mathcal{M}_4}\mathcal{M}'_4$$
and let  ${J}'=J\times_{\mathcal{M}_4}\mathcal{M}'_4$ be the Jacobian fibration  of the family $\mathcal{C}'\rightarrow \mathcal{M}'_4$ of genus $4$ curves. The rational map $\gamma'$
  provides a rational section of $\pi':J'\rightarrow \mathcal{M}'_4$ that we also denote by $\gamma'$. By construction, the variety $\mathcal{M}'_4$ has an involution $i$ which makes $\gamma'$ compatible with the involution $-{\rm Id}$ on ${J}$, that is,
\begin{eqnarray}\label{eqpourantiinvadi1ermai} \gamma'\circ i=-\gamma'.
\end{eqnarray}

The section $\gamma'$ gives (maybe with rational coefficients, and after replacing  $\mathcal{M}'_4$ by a dense Zariski open set) a divisor $D_{\gamma'}$ on $\mathcal{C}'$ and also on the induced relative second symmetric product ${\mathcal{C}'}^{(2)/\mathcal{M}'_4}$.  Hence we get a codimension $2$-cycle
$$D_{\gamma'}^2\in {\rm CH}^2({\mathcal{C}'}^{(2)/\mathcal{M}'_4}).$$
Let $q: {\mathcal{C}'}^{(2)/\mathcal{M}'_4}\rightarrow {\mathcal{C}}^{(2)/\mathcal{M}_4}$ be the natural double cover. Equation (\ref{eqpourantiinvadi1ermai}) tells that, over the same Zariski open set of  $\mathcal{M}'_4$, the divisor
 $D_{\gamma'}\in {\rm CH}^1({\mathcal{C}'}^{(2)/\mathcal{M}'_4})$ associated to $\gamma'$  is anti-invariant under the involution $q$. It follows that, over the same Zariski open set,
 $$D_{\gamma'}^2\in {\rm CH}^2({\mathcal{C}'}^{(2)/\mathcal{M}'_4})$$ is invariant under $q$, hence   the cycle
$D_{\gamma'}^2$
comes (with rational coefficients) from a codimension $2$-cycle on ${\mathcal{C}}^{(2)/\mathcal{M}_4}$. More precisely, possibly after shrinking $\mathcal{M}'_4$ and $\mathcal{M}_4$, one has by $q$-invariance
$$D_{\gamma'}^2=2q^*q_* D_{\gamma'}^2\,\,{\rm in}\,\,{\rm CH}^2({\mathcal{C}'}^{(2)/\mathcal{M}'_4}).$$

Applying Theorem \ref{proverni} to $q_* D_{\gamma'}^2$, one gets the following first step in the proof of Theorem \ref{theoquestionpirola}.
\begin{prop} \label{profirststep} Let $\mathcal{C}\rightarrow  \mathcal{M}_4$ be the universal curve of genus $4$, and ${J}\rightarrow \mathcal{M}_4 $ the associated Jacobian fibration. Then, if $\gamma$ is a rational section of the fibration $K=J/\pm {\rm Id}\rightarrow \mathcal{M}_4$, there exists a rational number $\lambda$ such that,  for a general point $b\in \mathcal{M}_4$,
\begin{eqnarray}\label{eq1ermai} D_{\gamma,b}^2=\lambda\, D_{\gamma_{P,b}}^2\,\,{\rm in}\,\,{\rm CH}_0(C_b^{(2)}).\end{eqnarray}
\end{prop}
\begin{rema}{\rm The descent to $J\rightarrow \mathcal{M}_4$ has been necessary only  to apply Proposition \ref{proverni} in order  to get (\ref{eq1ermai}). From now on, we will work on a degree $4$ cover $J''\rightarrow \mathcal{M}''_4$ of $J\rightarrow \mathcal{M}_4$, where both divisors $D_{\gamma_{P}}\in {\rm CH}^1(J'')$ and $D_{\gamma'}\in {\rm CH}^1(J'')$ are defined.}
\end{rema}
We now apply the theory of infinitesimal invariants described in Section \ref{secreview}. From Proposition \ref{profirststep} and Lemma \ref{lenouveau2804}, we get that, shrinking $\mathcal{M}''_4$ to a dense Zariski open set if necessary, the cycle $$D_{\gamma'}^2-\lambda D_{\gamma_P}^2\in
 {\rm CH}^2(J'')_\mathbb{Q}$$ is trivial, hence also its infinitesimal invariant
$$\delta D_{\gamma'}^2(b)-\lambda\, \delta D_{\gamma_P}^2(b)=([D_{\gamma'}^2]^{2,2}-\lambda \, [D_{\gamma_P}^2]^{2,2})_{\mid J''_b}$$
$$=([D_{\gamma'}]^{1,1}_{\mid J''_b})^2-\lambda \,([D_{\gamma_P}]^{1,1}_{\mid J''_b})^2\,\,{\rm in}\,\, H^2(\Omega^2_{J'' \mid J''_b})$$
at the general point $b\in \mathcal{M}''_4$.  The map $\mathcal{M}''_4\rightarrow \mathcal{M}_4$ is \'{e}tale at the general point of $\mathcal{M}''_4$, so the infinitesimal invariants
(see Section \ref{secreview})
$$[D_{\gamma'}]^{1,1}_{\mid J''_b}=\delta\gamma'(b)\in H^1(\Omega_{J'' \mid J''_b}),$$
$$ [D_{\gamma_P}]^{1,1}_{\mid J''_b}=\delta\gamma_P(b)\in H^1(\Omega_{J'' \mid J''_b}) $$
and their squares
are as  in the previous section,  in case (ii).
Using the explicit description (see (\ref{eqI1V}), (\ref{eqI2V2704})) of the spaces $I_{1,V}$ and $I_{2,V}$  of infinitesimal invariants given in Section \ref{secthsquare}, and the notation $\overline{M}_{2,V}$ of (\ref{eqM2barpourV}) for the square map,
one concludes the following.
\begin{coro}\label{corolecarrenul} With the assumptions and  notation of Proposition \ref{profirststep}, the infinitesimal invariant $\delta\gamma'(b)\in I_{1,V}$ of the section $\gamma'$ at a general point $b\in \mathcal{M}_4'$ satisfies
\begin{eqnarray} \overline{M}_{2,V}(\delta\gamma'(b))=\lambda  \overline{M}_{2,V}(\delta\gamma_{P}(b))\,\,{\rm in}\,\,I_{2,V}.
\end{eqnarray}
\end{coro}
\subsection{The equation $\overline{M}_{2,V}(\overline{\phi})=\lambda \,\overline{M}_{2,V}(\delta\gamma_P)$}
Let $C$ be a generic curve of genus $4$,  $V:=H^0(C,K_C)$  and let $q\in {\rm Sym}^2V$ be the unique (nondegenerate) quadratic form vanishing in $H^0(C,2K_C)$. According to Section \ref{secreview}, the Griffiths-Pirola normal function $\gamma_P$ (defined up to sign on $\mathcal{M}_4$), has an infinitesimal invariant
\begin{eqnarray}\label{eqinfinv2mai} \delta\gamma_{P,C}\in I_{1,V}\cong (V^*\otimes V_2)/V\end{eqnarray}
at $[C]\in\mathcal{M}_4$ (which is well-defined up to sign).
\begin{prop}\label{propcarreeqgalcarre} Let $C$ be a general curve of genus $4$ and let $\delta\gamma_{P,C}\in I_{1,V}$ be the infinitesimal invariant (\ref{eqinfinv2mai}). Then, if an element
$\overline{\phi}\in I_{1,V}$ satisfies
$$\overline{M}_{2,V}(\overline{\phi})=\lambda \,\overline{M}_{2,V}(\delta\gamma_{P,C}) \,\,{\rm in}\,\,I_{2,V},$$
one has either
\begin{eqnarray}\label{case1} \overline{M}_{2,V}(\overline{\phi})=0\,\, {\rm and}\,\,\lambda=0
\end{eqnarray}
or
\begin{eqnarray}\label{case2} \overline{\phi}= \sqrt{\lambda}\,\delta\gamma_{P,C}
\end{eqnarray}
for some square root $\sqrt{\lambda}$ of $\lambda$.
\end{prop}
\begin{rema}{\rm The meaning of Proposition \ref{propcarreeqgalcarre} is that the quadratic rational map
$$\overline{M}_{2,V}: \mathbb{P}(I_{1,V})\dashrightarrow \mathbb{P}(I_{2,V})$$  is, away from its indeterminacy locus, $1$-to-$1$ on its image over the   points of the form $\overline{M}_{2,V}(\delta\gamma_{P,C})$, with $C$ general.}
\end{rema}
The proof will use several lemmas.  We are going to use for this proof the lines contained in the quadric $Q=\{q=0\}$. There are two such families of lines parameterized by $P_1\cong \mathbb{P}^1,\,P_2\cong \mathbb{P}^1$. For each line $\Delta\subset Q$, there is
a corresponding vector subspace $S_\Delta\subset V$ of dimension $2$, and a quotient $E_\Delta$ of $V$ of dimension $2$, with
$$S_\Delta:=I_\Delta(1),\,E_\Delta=H^0(\Delta,\mathcal{O}_\Delta(1)).$$
We thus get two vector bundles $\mathcal{S}_i,\,\mathcal{E}_i$ on $P_i$, $i=1,\,2$, which are restricted from the Grassmannian $G(2,V)$. As we are considering lines contained in $Q$, on which $q=0$,  the construction in the previous section gives for $i=1,\,2$  linear maps
\begin{eqnarray}\label{eqrest1} \alpha_i: I_{1,V}\rightarrow H^0(P_i,\mathcal{H}om(\mathcal{S}_i,{\rm Sym}^2\mathcal{E}_i))
\end{eqnarray}
and linear maps
\begin{eqnarray}\label{eqrest2} \beta_i: I_{2,V}\rightarrow H^0(P_i,\mathcal{H}om(\bigwedge^2\mathcal{S}_i,\bigwedge^2({\rm Sym}^2\mathcal{E}_i))),
\end{eqnarray}
such that the following formula holds for any $\overline{\phi}\in I_{1,V}$:
\begin{eqnarray}\label{eqrestcarrre} \beta_i(\overline{M}_{2,V}(\overline{\phi}))=\wedge^2(\alpha_i(\overline{\phi}))\,\,{\rm in}\,\,H^0(P_i,\mathcal{H}om(\bigwedge^2\mathcal{S}_i,\bigwedge^2({\rm Sym}^2\mathcal{E}_i))).
\end{eqnarray}

We now have the following result.
\begin{lemm}\label{lerank2} If $C$ is a general curve of genus $4$, both morphisms  $$\alpha_i(\delta \gamma_{P,C})\in H^0(P_i,\mathcal{H}om(\mathcal{S}_i,{\rm Sym}^2\mathcal{E}_i)),  \,\,i=1,\,2,$$ have rank $2$ everywhere on $P_i$.
\end{lemm}
\begin{proof} Let $C\subset Q\subset \mathbb{P}^3$ be canonically embedded and let $t\in H^0(\mathbb{P}^3,\mathcal{O}_{\mathbb{P}^3}(3))$ be a cubic equation defining $C$ inside $Q$.
 We know by  Theorem \ref{theovernicubic} (see also \cite[Section 4.2]{verni} for more detail)  that $\delta\gamma_{P,C}\in I_{1,V}$ is given by the restrictions to $Q$ of the  partial derivatives of $t$, that is, the image $(\partial t)_{\mid Q}$ of $t$ in $V\otimes V_2$,  seen as an element of $V^*\otimes V_2$ thanks to the isomorphism $V\cong V^*$ given by $q$. Indeed,  if we change $t$ to $t+aq$, for some $a\in H^0(\mathbb{P}^3,\mathcal{O}_{\mathbb{P}^3}(1))$, then $(\partial t)_{\mid Q}$ is changed to

 $$(\partial (t+aq))_{\mid Q}=(\partial t)_{\mid Q}+a (\partial q)_{\mid Q},$$
 hence we get the same element in $I_{1,V}=(V^*\otimes V_2)/V$. Hence, $$t\mapsto \partial{t}_{\mid Q}\,\,{\rm mod.}\,\, V$$ provides a well defined $O(Q)$-equivariant inclusion of $H^0(Q,\mathcal{O}_Q(3))$ in $I_{1,V}=(V^*\otimes V_2)/V$, which is in fact unique.

 If $\Delta\subset Q$ is a member of the $i$-th rulling, then via the isomorphism $V\cong V^*$ given by $q$, the ideal $I_\Delta(1)$ is mapped to the space of  sections $H^0(\Delta,T_\Delta(-1))$. It follows that, for any point $[\Delta]\in P_i$,  $$\alpha_i(\delta\gamma_{P,C})_{\mid [\Delta]}\in {\rm Hom}(I_\Delta(1),H^0(\Delta,\mathcal{O}_\Delta(2)))$$
 identifies with the morphism
 $$H^0(\Delta,T_\Delta(-1))\rightarrow H^0(\Delta,\mathcal{O}_\Delta(2))$$
 given by the partial derivatives of $t_{\mid\Delta}$. The lemma then follows from the fact that the set of cubic polynomials on $\Delta$ with one vanishing partial derivative is a closed algebraic subset of  codimension $2$ in $H^0(\Delta,\mathcal{O}_\Delta(3))$.
\end{proof}
We will also need the following
\begin{lemm} \label{leinji12} The restriction map
$$(\alpha_1,\alpha_2): I_{1,V}\rightarrow H^0(P_1,\mathcal{H}om(\mathcal{S}_1,{\rm Sym}^2\mathcal{E}_1))\bigoplus H^0(P_2,\mathcal{H}om(\mathcal{S}_2,{\rm Sym}^2\mathcal{E}_2))$$ is injective, and its cokernel is isomorphic to $H^0(Q,\mathcal{O}_Q(3))$.
\end{lemm}
Before proving the lemma, let us explain how both spaces $H^0(P_i,\mathcal{H}om(\mathcal{S}_i,{\rm Sym}^2\mathcal{E}_i))$ map naturally to $H^0(Q,\mathcal{O}_Q(3))$. Let
$$\pi_i: Q\rightarrow P_i$$
be the natural morphisms (note that $Q\cong P_1\times P_2$ via $(\pi_1,\pi_2)$).
We have
\begin{eqnarray}\label{eqdun} H^0(P_i, \mathcal{H}om(\mathcal{S}_i,{\rm Sym}^2\mathcal{E}_i))=H^0(Q,\pi_i^*\mathcal{S}_i^*\otimes \mathcal{O}_Q(2)).
\end{eqnarray}
We now note that  we have an injection
\begin{eqnarray}\label{eqdedeux}j_i:\mathcal{O}_Q(-1)\hookrightarrow \pi_i^*\mathcal{S}_i
\end{eqnarray}
given by the Gauss map of $Q$, which produces at each point $x\in Q$  a linear form vanishing at $x$ and along the two lines in $Q$ passing through $x$.
Combining (\ref{eqdun}) and (\ref{eqdedeux}) gives us the desired maps
\begin{eqnarray}\label{eqetai} \eta_i: H^0(P_i,\mathcal{H}om(\mathcal{S}_i,{\rm Sym}^2\mathcal{E}_i))\rightarrow H^0(Q,\mathcal{O}_Q(3)).\end{eqnarray}
As we have the obvious equality of morphisms
$$j'_1\circ j_1=j'_2\circ j_2: \mathcal{O}_Q(-1)\rightarrow V^*\otimes \mathcal{O}_Q,$$
where $j'_i:\pi_i^*\mathcal{S}_i\hookrightarrow V^*\otimes \mathcal{O}_Q$, $i=1,\,2$,  are the natural inclusions (composed with the duality isomorphism $V\cong V^*$ given by $q$), one gets that, for any  $\overline{\phi}\in I_{1,V}$,
\begin{eqnarray}\label{eqeta1eta22105} \eta_1(\alpha_1(\overline{\phi}))=\eta_2(\alpha_2(\overline{\phi}))\,\,{\rm in}\,\,H^0(Q,\mathcal{O}_Q(3)).\end{eqnarray}

\begin{proof}[Proof of Lemma \ref{leinji12}] Let $\phi\in {\rm Hom}(V,V_2)$ with image $\overline{\phi}\in I_{1,V}$ and assume that $\alpha_i(\overline{\phi})=0$ in $H^0(P_i, \mathcal{H}om(\mathcal{S}_i,{\rm Sym}^2\mathcal{E}_i))$ for $i=1,\,2$. Then, defining the vector bundle $\mathcal{K}$ on $Q$ as the kernel of the evaluation map
${\rm ev}: V\otimes \mathcal{O}_Q\rightarrow \mathcal{O}_Q(1)$,
 the image of $\phi$ in $H^0(Q,V^*\otimes \mathcal{O}_Q(2))=V^*\otimes V_2$ vanishes in
$H^0(Q, \mathcal{K}^*(2))$, since by our assumption,  at any point $x$ of $Q$, it vanishes in ${\rm Hom}(I_\Delta,\mathcal{O}_x(2))$ and in ${\rm Hom}(I_{\Delta'},\mathcal{O}_x(2))$, where $\Delta$ and $\Delta'$ are the two lines contained in $Q$ and   passing through $x$, so that $I_x(1)=I_\Delta(1)\oplus I_{\Delta'}(1)$.
Dualizing the exact sequence
$$0\rightarrow \mathcal{K}\rightarrow V\otimes \mathcal{O}_Q\rightarrow \mathcal{O}_Q(1)\rightarrow 0,$$
we get that the image of $\phi$ in $H^0(Q,V^*\otimes \mathcal{O}_Q(2))$ comes from an element $\phi'$ of
$$H^0(Q,\mathcal{O}_Q(-1)\otimes \mathcal{O}_Q(2))=H^0(Q,\mathcal{O}_Q(1))=V.$$
Thus $\phi-\phi'$ vanishes in $H^0(Q,V^*\otimes \mathcal{O}_Q(2))$, or equivalently $\phi-\phi'=0$ in $V^*\otimes V_2$. As $\phi'\in V$, it follows that  $\overline{\phi}=0$ in $I_{1,V}=(V^*\otimes V_2)/V$, which proves the injectivity statement.
The other statement is then proved by  a dimension computation. Indeed, it is clear that \begin{eqnarray}\label{eqpourflecheeta}\eta_1-\eta_2:H^0(P_1,\mathcal{H}om(\mathcal{S}_1,{\rm Sym}^2\mathcal{E}_1))\bigoplus H^0(P_2,\mathcal{H}om(\mathcal{S}_2,{\rm Sym}^2\mathcal{E}_2))\rightarrow H^0(Q,\mathcal{O}_Q(3))\end{eqnarray} is surjective.
The right hand side in (\ref{eqpourflecheeta}) has dimension $16$ and the left hand side has dimension $48$.
  By (\ref{eqeta1eta22105}), the kernel of $\eta_1-\eta_2$ contains ${\rm Im}(\alpha_1,\alpha_2)$ which, as just proved, has dimension equal to ${\rm dim}\,((V^*\otimes V_2)/V)=32$. Hence the kernel of $\eta_1-\eta_2$ equals ${\rm Im}\,(\alpha_1,\alpha_2)$.
\end{proof}
\begin{proof}[Proof of Proposition \ref{propcarreeqgalcarre}] Let $C$ be a general curve of genus $4$, and let $\overline{\phi}\in I_{1,V}$ be such that
\begin{eqnarray}\label{eqducarregammaP} \overline{M}_{2,V}(\overline{\phi})=\lambda\, \overline{M}_{2,V}(\delta\gamma_{P,C}) \,\,{\rm in}\,\,I_{2,V}
\end{eqnarray}
for some coefficient $\lambda$. If $\lambda=0$, then (\ref{case1}) is satisfied, so we just have to study the case where $\lambda\not=0$.
We now use (\ref{eqrestcarrre}) and apply it to both $\overline{\phi}$ and $\delta\gamma_P$. We thus get from (\ref{eqducarregammaP}) that for $i=1,\,2$,

\begin{eqnarray}\label{eqducarregammaPnouveau} \wedge^2(\alpha_i(\overline{\phi}))=\lambda\,\wedge^2(\alpha_i(\delta\gamma_{P,C}))\,\,{\rm in}\,\,H^0(P_i,\mathcal{H}om(\bigwedge^2\mathcal{S}_i,\bigwedge^2({\rm Sym}^2\mathcal{E}_i))).
\end{eqnarray}
By Lemma \ref{lerank2}, $\alpha_i(\delta\gamma_{P,C})\in H^0(P_i,\mathcal{H}om(\mathcal{S}_i,{\rm Sym}^2\mathcal{E}_i))$ has rank $2$ everywhere, hence by (\ref{eqducarregammaPnouveau}), the same is true for the morphism $\alpha_i(\overline{\phi})$. Furthermore
 (\ref{eqducarregammaPnouveau}) implies  that the images of the two pointwise injective  morphisms $\alpha_i(\overline{\phi})\in H^0(P_i,\mathcal{H}om(\mathcal{S}_i,{\rm Sym}^2\mathcal{E}_i))$ and
$\alpha_i(\delta\gamma_P)$ are equal. Thus  we have, for $i=1,\,2$, a factorization
\begin{eqnarray}\label{eqaveclespsii}  \alpha_i(\overline{\phi})=\sqrt{\lambda} \,\alpha_i(\delta\gamma_{P,C})\circ \psi_i,
\end{eqnarray}
for some determinant $1$ automorphism $\psi_i$   of the vector bundle $\mathcal{S}_i$ on $P_i$.
Recalling (\ref{eqeta1eta22105}),
the automorphisms $\psi_1$ and $\psi_2$ have to satisfy the conditions that
\begin{eqnarray}\label{eqnewdu13oct} \eta_1(\alpha_1(\delta\gamma_{P,C})\circ \psi_1)=\eta_2(\alpha_2(\delta\gamma_{P,C})\circ \psi_2)\,\,{\rm in}\,\,H^0(Q,\mathcal{O}_Q(3)).
\end{eqnarray}
The proof of Proposition \ref{propcarreeqgalcarre} is now complete, by  Lemma \ref{leajoutedu13oct} proved below.
\end{proof}

\begin{lemm}\label{leajoutedu13oct} For a general genus $4$ curve $C\subset Q$,  equation (\ref{eqnewdu13oct}) implies the equalities  $$\psi_1=\mu\, {\rm Id}_{\mathcal{S}_1},\,\,\psi_2=\mu\, {\rm Id}_{\mathcal{S}_2},$$
where $\mu=\pm1$.
\end{lemm}

\begin{proof} We prove the lemma by an explicit computation for an explicit choice of equation $f_C$ for $C$. We observe first that  we have natural isomorphisms
$$\mathcal{S}_1^*\cong \mathcal{E}_1,\,\, \mathcal{S}_2^*\cong \mathcal{E}_2,$$
given by differentiating the equation $q$ of $Q$ along lines in $Q$.
Let $W_i:=H^0(P_i,\mathcal{O}_{P_i}(1))$. We have canonical isomorphisms
$$\mathcal{E}_1\cong W_2\otimes \mathcal{O}_{P_1}(1),\,\,\mathcal{E}_2\cong W_1\otimes \mathcal{O}_{P_1}(1).$$
The automorphisms $\psi_i$  are thus given by automorphisms of $W_i$, and we observe that equation (\ref{eqnewdu13oct}) is linear in the pair $(\psi_1,\psi_2)\in{\rm End}(W_1)\oplus {\rm End}(W_2)$. Proving for some $C$ that the only solutions of  (\ref{eqnewdu13oct}) in ${\rm End}(W_1)\oplus {\rm End}(W_2)$  are  given by the pairs $(\mu \,{\rm Id}_{W_1},\mu\,{\rm Id}_{W_2})$ will thus  imply the same statement for a generic $C$. The number $\mu$ has then to be $\pm1$ since $\psi_i$ are of determinant $1$.

As we have $$\mathcal{S}_1^*\times {\rm Sym}^2\mathcal{E}_1\cong \mathcal{E}_1\otimes {\rm Sym}^2\mathcal{E}_1\cong W_{2}\otimes  {\rm Sym}^2W_{2}\otimes \mathcal{O}_{P_1}(3),$$
and similarly with $1$ and $2$ exchanged, 
we get
$$H^0(P_1,\mathcal{S}_1^*\times {\rm Sym}^2\mathcal{E}_1)\cong W_2\otimes {\rm Sym}^2W_2\otimes {\rm Sym}^3W_1,$$
$$ H^0(P_2,\mathcal{S}_1^*\times {\rm Sym}^2\mathcal{E}_2)\cong W_1\otimes {\rm Sym}^2W_1\otimes {\rm Sym}^3W_2.$$
Furthermore, trivializing ${\rm det}\,W_i$, there are canonical splitting
$$  W_1\otimes {\rm Sym}^2W_1\cong W_1\oplus {\rm Sym}^3W_1,\,\,W_2\otimes {\rm Sym}^2W_2\cong W_2\oplus {\rm Sym}^3W_2.
$$
We thus get natural embeddings of ${\rm Sym}^3W_1\otimes {\rm Sym}^3W_2=H^0(Q,\mathcal{O}_Q(3))$ in $H^0(P_i,\mathcal{S}_i^*\times {\rm Sym}^2\mathcal{E}_i)$. Playing on the action of the group $G={\rm SL}(W_2)\times {\rm SL}(W_1)$ on these spaces, one checks that these embeddings agree with the restriction of $\alpha_i$ to the  subspace $H^0(Q,\mathcal{O}_Q(3))\subset I_{1,V}$ which is a canonical $G$-subrepresentation (see \cite{verni}). Similarly, we get natural morphisms
$$H^0(P_i,\mathcal{S}_i^*\times {\rm Sym}^2\mathcal{E}_i)\rightarrow {\rm Sym}^3W_1\otimes {\rm Sym}^3W_2=H^0(Q,\mathcal{O}_Q(3)),$$  which, again by $G$-equivariance, agree with $\eta_i$. We now use the fact, already used above and  proved in \cite{verni}, that $ \delta\gamma_{P,C}$ belongs to the $G$-summand $H^0(Q,\mathcal{O}_Q(3))\subset I_{I,V}$ and is a nonzero multiple of $f_C$.

We now choose 3 elements $\alpha\otimes \beta,\,\alpha'\otimes \beta',\,\alpha''\otimes \beta''$ in $W_1\otimes W_2$ and   set
$$f_C:=\alpha^3\otimes\beta^3+ {\alpha'}^3\otimes{\beta'}^3+{\alpha''}^3\otimes{\beta''}^3\in {\rm Sym}^3W_1\otimes {\rm Sym}^3W_2\subset I_{1,V}.$$
We have
$$\alpha_1(f_C)=  3(\alpha^3\otimes\beta\otimes \beta^2+ {\alpha'}^3\otimes\beta'\otimes {\beta'}^2+{\alpha''}^3\otimes\beta''\otimes {\beta''}^2)\,\,{\rm in}\,\, {\rm Sym}^3W_1\otimes W_2\otimes {\rm Sym}^2W_2 ,$$
$$\alpha_2(f_C)= 3(\alpha\otimes\alpha^2\otimes\beta^3+ \alpha'\otimes {\alpha'}^2\otimes{\beta'}^3+\alpha''\otimes{\alpha''}^2\otimes{\beta''}^3) \,\,{\rm in}\,\, W_1\otimes {\rm Sym}^2W_1\otimes  {\rm Sym}^3W_2.$$
It follows that we have
\begin{eqnarray}\label{eqdu13avecgroscal}\alpha_1(f_C)\circ \psi_1=3(\alpha^3\otimes (\beta\circ\psi_1)\otimes \beta^2+ {\alpha'}^3\otimes(\beta'\circ\psi_1)\otimes {\beta'}^2+{\alpha''}^3\otimes(\beta''\circ\psi_1)\otimes {\beta''}^2)\,\, \end{eqnarray}
in $ {\rm Sym}^3W_1\otimes W_2\otimes {\rm Sym}^2W_2$ and similarly
\begin{eqnarray}\label{eqdu13avecgroscalprime}
\alpha_2(f_C)\circ \psi_2= 3((\alpha\circ \psi_2)\otimes\alpha^2\otimes\beta^3+ (\alpha'\circ\psi_2)\otimes {\alpha'}^2\otimes{\beta'}^3+(\alpha''\circ\psi_2)\otimes{\alpha''}^2\otimes{\beta''}^3)
\end{eqnarray}  in $ W_1\otimes {\rm Sym}^2W_1\otimes  {\rm Sym}^3W_2$.

Applying $\eta_i$ to (\ref{eqdu13avecgroscal}), (\ref{eqdu13avecgroscalprime}), equation (\ref{eqnewdu13oct}) applied to $f_C=\delta\gamma_{P,C}$ provides the equality
\begin{eqnarray} \label{eqdu13avecgroscalrelationcubescarres}
\alpha^3\otimes (\beta\circ\psi_1) \beta^2+ {\alpha'}^3\otimes(\beta'\circ\psi_1) {\beta'}^2
+{\alpha''}^3\otimes(\beta''\circ\psi_1) {\beta''}^2 \\
\nonumber=
(\alpha\circ \psi_2)\alpha^2\otimes\beta^3+ (\alpha'\circ\psi_2) {\alpha'}^2\otimes{\beta'}^3+(\alpha''\circ\psi_2){\alpha''}^2\otimes{\beta''}^3
\end{eqnarray}
 in $ {\rm Sym}^3W_1\otimes  {\rm Sym}^3W_2$.
We rewrite (\ref{eqdu13avecgroscalrelationcubescarres}) as
\begin{eqnarray} \label{eqdu13avecgroscalrelationcubescarresrelation}
(\alpha^2\otimes \beta^2)(\alpha\otimes (\beta\circ\psi_1)-(\alpha\circ \psi_2)\otimes \beta)+ ({\alpha'}^2\otimes {\beta'}^2)(\alpha'\otimes (\beta'\circ\psi_1) -(\alpha'\circ \psi_2)\otimes \beta')
\\
\nonumber
+({\alpha''}^2\otimes {\beta''}^2)(\alpha''\otimes (\beta''\circ\psi_1)-(\alpha''\circ \psi_2)\otimes \beta'')=0 \,\,{\rm in}\,\, {\rm Sym}^3W_1\otimes  {\rm Sym}^3W_2.
\end{eqnarray}
Choosing $\alpha,\,\beta,\,\alpha',\,\beta',\,\alpha'',\,\beta''$ in general position, the three sections
$ \alpha^2\otimes \beta^2,\,{\alpha'}^2\otimes {\beta'}^2,\,{\alpha''}^2\otimes {\beta''}^2\in {\rm Sym}^2W_1\otimes {\rm Sym}^2W_2=H^0(Q,\mathcal{O}_Q(2))$ form a regular sequence on $Q$, hence  (\ref{eqdu13avecgroscalrelationcubescarresrelation}) implies by a Koszul argument the following equalities in $W_1\otimes W_2$.
\begin{eqnarray} \label{eqdu13avecgroscalrelationcubescarresrelationlineaire}\alpha\otimes (\beta\circ\psi_1)-(\alpha\circ \psi_2)\otimes \beta=0,\,\,\alpha'\otimes (\beta'\circ\psi_1)-(\alpha'\circ \psi_2)\otimes \beta'=0,\\
\nonumber
\alpha''\otimes (\beta''\circ\psi_1)-(\alpha''\circ \psi_2)\otimes \beta''=0.\end{eqnarray}
Equations (\ref{eqdu13avecgroscalrelationcubescarresrelationlineaire}) immediately  imply  $\psi_1=\mu\,{\rm Id}_{W_2},\, \psi_2=\mu\,{\rm Id}_{W_1}$ for some coefficient $\mu$.
\end{proof}
\subsection{Proof of Theorem \ref{theoquestionpirola}}
We complete in this section the proof of Theorem \ref{theoquestionpirola}. We first establish a slightly weaker     statement.
\begin{theo}\label{theoquestionpirolatexte} Let $\mathcal{C}\rightarrow  \mathcal{M}_4$ be the universal curve of genus $4$, and ${J}\rightarrow \mathcal{M}_4 $ the associated Jacobian fibration. Then all rational sections of the fibration $K=J/\pm {\rm Id}\rightarrow \mathcal{M}_4$ are given by  rational  multiples of the Griffiths-Pirola normal function $\gamma_P$.
\end{theo}
The meaning of this statement is  that,  any  rational section $\gamma$ satisfies a relation \begin{eqnarray}\label{eqstarfinfin7mai} N\gamma=N'\gamma_P\end{eqnarray} as rational sections of $K$,  for some  integers $N$, $N'$ with $N\not=0$.
\begin{proof}[Proof of Theorem \ref{theoquestionpirolatexte}] By Proposition \ref{profirststep} and its Corollary \ref{corolecarrenul}, given a rational section $\gamma$ of ${K}\rightarrow \mathcal{M}_4$,  the infinitesimal invariant $\delta{\gamma}'_C\in I_{1,V}$ at a general point $[C]$ of $\mathcal{M}''_4$  of the corresponding  normal function ${\gamma}'$ defined on an adequate cover $\mathcal{M}''_4$ of $\mathcal{M}_4$, (on which  the normal function  $\gamma_P\in \Gamma(\mathcal{J}^\vee)$ is also defined), satisfies the hypotheses of Proposition \ref{propcarreeqgalcarre} for some rational number $\lambda$ for which (\ref{eq1ermai}) holds. It follows that either (\ref{case1}) or (\ref{case2}) is satisfied by $\overline{\phi}:=\delta{\gamma}'_C$. If (\ref{case1}) is satisfied, then the coefficient $\lambda$ vanishes, hence by (\ref{eq1ermai}) the divisor $D_\gamma\in {\rm Pic}^0(C)={\rm Pic}^0(C^{(2)})={\rm CH}^1(C^{(2)})$ satisfies
\begin{eqnarray}\label{eqDgammacarrevan} D_\gamma^2=0\,\,{\rm in}\,\,{\rm CH}_0(C^{(2)}).
\end{eqnarray}
By Theorem \ref{theoabelian}(ii), it follows that $D_\gamma$ is a torsion element of ${\rm Pic}^0(C)$ for very general $C$, hence by a Baire countability argument, there exists an integer $N\not=0$ such that (\ref{eqstarfinfin7mai}) is satisfied for any $C$,   with $N'=0$.

It remains to study what happens if (\ref{case2}) is satisfied. In this case, $\lambda\not=0$ and  the infinitesimal invariants
$\delta{\gamma}', \,\delta{\gamma}_P\in \Gamma((\mathcal{H}^{0,1}\otimes\Omega_{\mathcal{M}_4''})/{\rm Im}\overline{\nabla})$  of ${\gamma}'$ and   $\gamma_P$ respectively, satisfy the relation
\begin{eqnarray}\label{eqgammagammaPinf}
\delta{\gamma}'=\sqrt{\lambda}\, \delta{\gamma}_P.
\end{eqnarray}
Recall that these infinitesimal invariants are computed locally in the Euclidean topology by choosing local  lifts $\tilde{\gamma}',\,\tilde{\gamma}_P$ in $\Gamma(\mathcal{H}^1)$, and defining
$$\delta {\gamma}':=(\nabla \tilde{\gamma}')^{0,1} \,\,{\rm mod}\,\,\overline{\nabla}\mathcal{H}^{1,0},\,$$
$$\delta {\gamma}_P:=(\nabla \tilde{\gamma}_P)^{0,1} \,\,{\rm mod}\,\,\overline{\nabla}\mathcal{H}^{1,0}.$$
The local lifts are defined up to the addition of local sections of $\mathcal{H}^{1,0}$ and  $H^1_\mathbb{Z}$, making the definition above independent of the local lift.
From this description and (\ref{eqgammagammaPinf}), we conclude  that there exist local lifts  $\tilde{\gamma}'$, resp. $\tilde{\gamma}_P$ of $\gamma'$, resp. $\gamma_P$,  satisfying
\begin{eqnarray}\label{eqlocalliftsqrt} (\nabla( \tilde{\gamma}'-\sqrt{\lambda}\,\tilde{\gamma}_P))^{0,1}=0\,\,{\rm in} \,\, \Gamma(\mathcal{H}^{0,1}\otimes \Omega_{\mathcal{M}''_4}).
\end{eqnarray}
Arguing as in the proof of Lemma \ref{lepourlatorsion}, we find that the  local lifts $\tilde{\gamma}'-\sqrt{\lambda}\,\tilde{\gamma}_P$ satisfying (\ref{eqlocalliftsqrt}) are in fact horizontal, and  unique modulo  sections of the local system $H^1_{\mathbb{Q}(\sqrt{\lambda})}$. (Due to the coefficient $\sqrt{\lambda}$ in (\ref{eqlocalliftsqrt}), we need to enlarge the coefficients of the local system here.) This is equivalent to saying that
there  exists a global  section $$\nu\in \Gamma(\mathcal{M}''_4,H^1_\mathbb{C}/H^1_{\mathbb{Q}(\sqrt{\lambda})})$$ whose image in $
\Gamma(\mathcal{M}''_4,\mathcal{H}^{0,1}/H^1_{\mathbb{Q}(\sqrt{\lambda})})$ equals $\gamma'-\sqrt{\lambda}\,\gamma_P$. Arguing as in the proof of Lemma \ref{lepourlatorsion}, we get that $\nu=0$, so that
\begin{eqnarray}\label{eqfindu3mai} \gamma'-\sqrt{\lambda}\,\gamma_P=0\,\,{\rm in}\,\,  \Gamma(\mathcal{M}''_4,\mathcal{H}^{0,1}/H^1_{\mathbb{Q}(\sqrt{\lambda})}).
\end{eqnarray}
To conclude the proof of Theorem \ref{theoquestionpirolatexte}, we just need to prove  that $\sqrt{\lambda} $ is a rational number. This is done as follows: given a family $J\rightarrow B$ of abelian varieties and  a normal function $\nu\in \Gamma(B,\mathcal{J}^\vee)=\Gamma(B,\mathcal{H}^{0,1}/H^1_\mathbb{Z})$,  the class $$[\nu ] \in H^1(B, H^1_\mathbb{Z})$$
of $\nu$, introduced by Griffiths, is obtained by writing the exact sequence (\ref{eqexactaveczenbas})
$$0\rightarrow H^1_\mathbb{Z}\rightarrow  \mathcal{H}^{0,1}\rightarrow \mathcal{J}^\vee\rightarrow 0$$
and by applying to $\nu$ the associated long exact sequence.

If $J$ is the Jacobian fibration of a family of curves $\mathcal{C}\rightarrow B$ and  $\nu=\nu_D$ comes from a relative divisor $D\in{\rm Pic}^0(\mathcal{C}/B)$, this class is nothing but the class $[D]$ of the divisor $D$, at least with $\mathbb{Q}$-coefficients, for which we have $$H^1(B,H^1_\mathbb{Q})\subset H^2(\mathcal{C},\mathbb{Q})/H^2(B,\mathbb{Q}).$$ Coming back to our normal function $\gamma_P$ for  the family $\mathcal{C}''\rightarrow \mathcal{M}_4''$,  the class
$$[\gamma_P]\in H^2(\mathcal{C}'',\mathbb{Q})/H^2(\mathcal{M}_4'',\mathbb{Q})$$ does not vanish. Indeed, this would imply otherwise by  a standard argument that the normal function $\gamma_P$ is trivial.  The condition $\gamma'-\sqrt{\lambda}\,\gamma_P=0$ of (\ref{eqfindu3mai})  implies  that
the two  cohomology classes
$[\gamma']$ and  $[\gamma_P]$   satisfy  $$[\gamma']-\sqrt{\lambda}\,[\gamma_P]=0\,\,{\rm  in }\,\, H^2(\mathcal{C}'',\mathbb{Q}(\sqrt{\lambda}))/H^2(\mathcal{M}_4'',\mathbb{Q}(\sqrt{\lambda})).$$
As $[\gamma']$ and  $[\gamma_P]$ are rational, that is, belong to $H^2(\mathcal{C}'',\mathbb{Q})/H^2(\mathcal{M}_4'',\mathbb{Q})$, and $[\gamma_P]\not=0$, we get that
  $\sqrt{\lambda}$ is rational and the proof of Theorem \ref{theoquestionpirolatexte} is finished.
\end{proof}
The proof of Theorem \ref{theoquestionpirola} is then concluded by the following two lemmas.

\begin{lemm} Assumptions and notation being as before, the rational number $\lambda$ of (\ref{eqgammagammaPinf}) is an integer.
\end{lemm}
\begin{proof} Let $\gamma$ be a rational section of $K\rightarrow\mathcal{M}_4$. We restrict $\gamma$ over a rational curve in $\mathcal{M}_4$ constructed as follows: let $Q\subset \mathbb{P}^3$ be a smooth quadric and let
$T_0,\,T_\infty \in H^0(Q,\mathcal{O}_Q(3))$ generate a Lefschetz pencil $\mathbb{P}^1=\langle T_0,\,T_\infty\rangle$. The smooth fibers $C_t:=\{T_t=0\}$ form a $1$-parameter family of genus $4$ curves, hence we get a curve $\mathbb{P}^1=B\supseteq B^0\subset \mathcal{M}_4$. The Lefschetz pencil provides a morphism
$$\pi_T:\widetilde{Q}\rightarrow B,$$
where $\widetilde{Q}$ is the quadric $Q$ blown-up along its base locus,
and we get a family
\begin{eqnarray}\label{eqfamilyffinfin}\pi_T^0:\widetilde{Q}^0\rightarrow B^0,\end{eqnarray}
by restricting $\pi_T$ over the regular locus $B^0$ of $ \pi_T$.

We observe now that on the family (\ref{eqfamilyffinfin}), the Griffiths-Pirola normal function is well-defined (there is no need to pass to a double cover), by ordering the  two rullings $\delta_1,\,\delta_2$ of $Q$, and by defining
$$\gamma_P(t)=(\delta_1-\delta_2)_{\mid \widetilde{Q}^0_t}\in{\rm Pic}^0(\widetilde{Q}^0_t)$$
for $t\in B^0$.
We claim  that the section $\gamma$ restricted to $B^0$ also lifts to a section of the  Jacobian fibration $J_{\pi_T}\rightarrow B^0$.
Indeed, the section $\gamma$ restricted to $B^0$ lifts in any case  to a section $\gamma'$ of the pulled-back Jacobian fibration
$J_{\pi_T}'\rightarrow {B'}^0$,  for some  double cover ${B'}^0\rightarrow B^0$.
On any irreducible  component of  this double cover, we have by Theorem \ref{theoquestionpirolatexte} that
$$\gamma'=\lambda\gamma_P. $$
On the other hand, we have
by construction of the double cover ${B'}^0$:
$$\gamma'\circ i=-\gamma',$$
where $i$ is the involution of ${B'}^0$ over $B^0$,
while $\gamma_P\circ i=\gamma_P$. Assuming the section $\gamma'$ is  not of torsion (equivalently is  nonzero by Lemma \ref{letorsion4mai} below), this is possible only if the double cover ${B'}^0$ has two components, so that $\gamma'$ is well defined on each component isomorphic to $B^0$. This proves the claim.
The two sections $\gamma_P $ and $\gamma'$ of the Jacobian fibration $J_{\pi_T}\rightarrow B^0$ have cohomology classes
$$[\gamma_P],\,[\gamma']\in H^2(\widetilde{Q}^0,\mathbb{Z})$$
which satisfy
the relation
$$[\gamma']=\lambda [\gamma_P],\,\in H^2(\widetilde{Q}^0,\mathbb{Q}).$$
That $\lambda$ is an integer then follows from the fact that  the class
$[\gamma_P]\in H^2(\widetilde{Q}_T^0,\mathbb{Z})$ is nonzero and  primitive (that is, not divisible by an integer $>1$) in $H^2(\widetilde{Q}_T^0,\mathbb{Z})$.
\end{proof}
\begin{lemm}\label{letorsion4mai} Any torsion section of the Kummer fibration $K\rightarrow \mathcal{M}_4$ is the zero section.
\end{lemm}
\begin{proof} A torsion section of $K\rightarrow \mathcal{M}_4$ is  given  by local sections of the local system  $H^1_{\mathbb{Z}/N}$ on $\mathcal{M}_4$, that glue up to sign. Given a point $t\in \mathcal{M}_4$, this would provide equivalently an element $\gamma_t$ of  $H^1(\mathcal{C}_t,\mathbb{Z}/N\mathbb{Z})$ such that for any $\eta\in\pi_1(\mathcal{M}_4,t)$,
\begin{eqnarray}\label{eqfinifin4mai}\rho(\eta)(\gamma_t)=\pm \gamma_t.\end{eqnarray}
As the monodromy representation $\rho$ with integral coefficients has for image the full symplectic group of
$H^1(\mathcal{C}_t,\mathbb{Z})$ equipped with the intersection form,  condition (\ref{eqfinifin4mai}) implies that $\gamma_t=0$.
\end{proof}


\begin{thebibliography}{99}
\bibitem{bakker}  B. Bakker, J.  Tsimerman. Functional transcendence of periods and the geometric Andr\'{e}-Grothendieck period conjecture. Forum Math. Sigma 13 (2025), Paper No. e97.
\bibitem{baldi} G. Baldi, D. Urbanik. Effective atypical intersections and applications to orbit closures, arXiv:2406.16628.
\bibitem{beauville} A. Beauville. Sur l'anneau de Chow d'une vari\'{e}t\'{e} ab\'{e}lienne.  Math. Ann. 273 (1986), no. 4, 647-651.
\bibitem{beauville1} A. Beauville.  Quelques remarques sur la transformation de Fourier dans l'anneau de Chow d'une vari\'{e}t\'{e} ab\'{e}lienne.  {\it Algebraic geometry (Tokyo/Kyoto, 1982)}, 238-260,
Lecture Notes in Math., 1016, Springer, Berlin, (1983).
\bibitem{blochabelian} S. Bloch. Some elementary theorems about algebraic cycles on Abelian varieties. Invent. Math. 37 (1976), no. 3, 215-228.
    \bibitem{nagloo} D. Bl\'{a}zquez-Sanz, G. Casale, J. Freitag, J. Nagloo. A differential approach to Ax-Schanuel, I. Arxiv:2102.03384v4, to appear in the Annals of Mathematics.
\bibitem{bloch} S. Bloch.  Torsion algebraic cycles and a theorem of Roitman. Compositio Math. 39 (1979), no. 1, 107-127.
    \bibitem{blochsri}  S. Bloch, V. Srinivas. Remarks on
 correspondences and algebraic cycles,
 Amer. J. of Math. 105 (1983) 1235-1253.
\bibitem{colombopirola}  E.  Colombo, J. C. Naranjo,  G. P. Pirola. On the dimension of Voisin sets in the moduli space of abelian varieties. Math. Ann. 381 (2021), no. 1-2, 91-104.
\bibitem{greengriffiths} M. Green, Ph.  Griffiths.  An interesting 0-cycle. Duke Math. J. 119 (2003), no. 2, 261-313.
    \bibitem{deligne} P. Deligne. Th\'{e}orie de Hodge. II. (French) Inst. Hautes \'{E}tudes Sci. Publ. Math. No. 40 (1971), 5–57.
        \bibitem{fulavi}  L. Fu, R.  Laterveer, Ch.  Vial.
The generalized Franchetta conjecture for some hyper-K\"{a}hler varieties.
(With an appendix by the authors and Mingmin Shen.)
J. Math. Pures Appl. (9) 130 (2019), 1-35.
\bibitem{griffiths} Ph. Griffiths.   Infinitesimal variations of Hodge structure. III. Determinantal varieties and the infinitesimal invariant of normal functions. Compositio Math. 50 (1983), no. 2-3, 267-324.

    \bibitem{mazurmessing} B. Mazur, W. Messing.
{\it Universal extensions and one dimensional crystalline cohomology}.
Lecture Notes in Mathematics, Vol. 370. Springer-Verlag, Berlin-New York, (1974).
\bibitem{mestrano} N. Mestrano.  Conjecture de Franchetta forte.  Invent. Math. 87 (1987), no. 2, 365-376.
\bibitem{mukai} S. Mukai. Duality between $D(X)$ and $D(\widehat{X})$ with its application to Picard sheaves. Nagoya Math. J. 81 (1981), 153-175.
    \bibitem{nitsure} N. Nitsure. Deformation theory for vector bundles. in {\it  Moduli spaces and vector bundles}, 128–164,
London Math. Soc. Lecture Note Ser., 359, Cambridge Univ. Press, Cambridge, (2009).
\bibitem{nori} M.  Nori. Algebraic cycles and Hodge-theoretic connectivity. Invent. Math. 111 (1993), no. 2, 349-373.
\bibitem{roitman} A. A. Rojtman.  The torsion of the group of 0-cycles modulo rational equivalence. Ann. of Math. (2) 111 (1980), no. 3, 553-569.
    \bibitem{simpson} C. Simpson. Moduli of representations of the fundamental group of a smooth projective variety. II.
Inst. Hautes \'{E}tudes Sci. Publ. Math. No. 80 (1994), 5-79 (1995).
\bibitem{verni} M. Verni. On infinitesimal invariants of genus $4$ curves and the Griffiths-Pirola cycle,  arXiv:2607.12793.
\bibitem{voisinhodgebloc} C. Voisin. The generalized Hodge and Bloch conjectures are equivalent for general complete intersections, Annales scientifiques de l'ENS 46, fascicule 3 (2013), 449-475.
\bibitem{voisinincras} C. Voisin. Une remarque sur l'invariant infinit\'{e}simal des fonctions normales. C. R. Acad. Sci. Paris S\'{e}r. I Math. 307 (1988), no. 4, 157-160.
\bibitem{voisinchowab} C. Voisin. Chow ring and gonality of general abelian varieties. Ann. H. Lebesgue 1 (2018), 313-332.
    \bibitem{voisinsurfaces} C. Voisin. Variations de structure de Hodge et z\'{e}ro-cycles sur les surfaces g\'{e}n\'{e}rales.  Math. Ann. 299 (1994), no. 1, 77-103.
\bibitem{voisinchow} C. Voisin. {\it Chow rings, decomposition of the diagonal, and the topology of families}. Annals of Mathematics Studies, 187. Princeton University Press, Princeton, NJ, 2014. viii+163 pp. ISBN: 978-0-691-16051-1; 0-691-16051-1.
    \bibitem{voisinbook} C. Voisin.  {\it Hodge theory and complex algebraic geometry.} II.
Translated from the French by Leila Schneps. Cambridge Studies in Advanced Mathematics, 77. Cambridge University Press, Cambridge, (2003).
    \bibitem{voisinICM94} C. Voisin.  Variations of Hodge structure and algebraic cycles. {\it Proceedings of the International Congress of Mathematicians}, Vol. 1, 2 (Z\"{u}rich, 1994), 706-715, Birkh\"{a}user, Basel, 1995.
        \bibitem{hangxue} H. Xue. A quadratic point on the Jacobian of the
universal genus four curve, Math. Res. Lett., Volume 22, Number 5, 1563-1571, (2015).
\bibitem{yinq} Q. Yin. The generic nontriviality of the Faber-Pandharipande cycle.
Int. Math. Res. Not. IMRN 2015, no. 5, 1263-1277.
\end{thebibliography}
\end{document}